\documentclass[a4paper]{amsart}
\usepackage{enumerate, amsfonts, amssymb, amsthm, mathrsfs, mathabx, booktabs}

\usepackage{tikz}
\usetikzlibrary{shapes,decorations.pathmorphing,calc,arrows}

\usepackage[para]{footmisc}

\usepackage[colorlinks=true]{hyperref}
\numberwithin{equation}{section}
\usepackage{cleveref}

\newcommand{\Dfn}[1]{\emph{\color{blue}#1}} 
\newcommand{\FindStat}[1]{\url{www.findstat.org/#1}}
\newcommand{\OEIS}[1]{\url{www.oeis.org/#1}}

\newcommand{\coc}{d}
\newcommand{\coa}{d}

\newcommand{\ind}{\textbf{b}}

\numberwithin{equation}{section}
\theoremstyle{plain}
\newtheorem{lemma}[equation]{Lemma}
\newtheorem{theorem}[equation]{Theorem}

\newtheorem{corollary}[equation]{Corollary}
\newtheorem{proposition}[equation]{Proposition}

\theoremstyle{definition}
\newtheorem{definition}[equation]{Definition}
\newtheorem{remark}[equation]{Remark}

\newtheorem{example}[equation]{Example}

\crefname{proposition}{Prop.}{Props.}
\crefname{corollary}{Cor.}{Cors.}
\crefname{theorem}{Thm.}{Thms.}
\crefname{equation}{Eq.}{Eqs.}

\Crefname{proposition}{Proposition}{Propositions}
\Crefname{theorem}{Theorem}{Theorems}
\Crefname{corollary}{Corollary}{Corollaries}
\Crefname{equation}{Equation}{Equations}

\newcommand\Dyck{\mathcal{D}}    
\newcommand\DyckP{\mathcal{D}^{\mathrm 0}} 
\newcommand\DyckH{\mathcal{D}^{\mathrm r}} 
\newcommand{\LK}{{\rm LK}}    
\newcommand{\LKP}{{\rm LK}^0} 
\newcommand{\LKH}{{\rm LK}^r} 
\DeclareMathOperator{\BJS}{BJS}  

\DeclareMathOperator{\End}{End}%
\DeclareMathOperator{\Ext}{Ext}%
\DeclareMathOperator{\Hom}{Hom}%
\DeclareMathOperator{\pd}{pd}%
\DeclareMathOperator{\gd}{gd}%

\newcommand{\ZZ}{\mathbb{Z}}
\newcommand{\NN}{\mathbb{N}}
\newcommand{\FF}{\mathbb{F}}

\newcommand{\defeq}{:=}

\definecolor{lightgrey}{rgb}{0.7,0.7,0.7}

\makeatletter
\long\def\ifnodedefined#1#2#3{%
    \@ifundefined{pgf@sh@ns@#1}{#3}{#2}%
}
\makeatother

\newcommand{\ARquiver}[1]{
  \foreach \x\y\c\lab in {#1}{
    \coordinate (A\x-\y) at (\x,\y);
    \draw[fill=black] (A\x-\y) circle (.05);
    \pgfmathtruncatemacro{\xplusone} {\x+1}
    \pgfmathtruncatemacro{\xminusone}{\x-1}
    \pgfmathtruncatemacro{\yminusone}{\y-1}

    \node[anchor=north] at ($(A\x-\y)-(0,0.1)$) {\tiny$\lab$};
    \ifnodedefined{A\xplusone-\yminusone}{
      \draw[->,shorten <=7pt, shorten >=7pt] (A\x-\y) -- (A\xplusone-\yminusone);
    }{}
    \ifnodedefined{A\xminusone-\yminusone}{
      \draw[->,shorten <=7pt, shorten >=7pt] (A\xminusone-\yminusone) -- (A\x-\y);
    }{}
  }
}

\newcommand{\drawPath}[5]
{%
  \draw[rounded corners=1, line width=#5, #3] #4
  \foreach \dir in {#1}%
  {
    \ifnum\dir=#2
    -- ++(1,0)
    \else
    -- ++(0,1)
    \fi
  };}
\newcommand{\drawLK}[1]{%
  \foreach \x/\y/\xx/\yy in {#1}
  {
    \draw[line width=0.5, color=green, decorate, decoration={snake, segment length=5, amplitude=1pt}] (\y,\x) -- (\y,\xx) -- (\yy,\xx);
  };}
\newcommand{\drawPeaksValleys}[1]{%
  \foreach \y/\x in {#1}
  {
    \node at (\x-0.5,\y-0.5) {\huge$\times$};
  };}
\newcommand{\drawCycleDiagram}[3]{%
  \foreach \y/\x in {#1}
  {
    \draw[line width=0.5, color=red, decorate, decoration={snake, segment length=5, amplitude=1pt}] (\x-0.5,\x-0.5) -- (\x-0.5,\y-0.5) node[color=black] {\huge$\times$} -- (\y-0.5,\y-0.5);
  };
  \foreach \x in {#2}
  {
    \draw[line width=1.5] (\x-0.75,\x-0.25) -- (\x-0.75,\x-0.75) -- (\x-0.25,\x-0.75);
  }
  \foreach \x in {#3}
  {
    \node[line width=1.5, circle, draw] at (\x-0.5,\x-0.5) {};
  }
}
\newcommand{\drawRegular}[4]{%
  \foreach \x/\y in {#1}
  {
    \node[fill=red, circle] at (\y,\x) {};
  }
  \foreach \x/\y in {#2}
  {
    \node[fill=red, diamond] at (\y,\x) {};
  }
  \foreach \c [count=\n from -1] in {#4}{};
  \foreach \c [count=\x from 0] in {#3}
  {
    \node at (14.5, \n-\x+1.5) {$d_{\x}$};
    \node at (15.25, \x+0.5) {\c};
  };
  \foreach \c [count=\y from 0] in {#4}
  {
    \node at (\n-\y+0.5, -0.5) {$c_{\y}$};
    \node at (\n-\y+0.5, -1.25) {\c};
  };
}

\newcommand{\drawLabel}[4]{%
  \foreach \c [count=\n from -1] in {#2}{};
  \foreach \c [count=\y from  0] in {#2}
  {
    \node at (-1.25,\n-\y+0.5) {$#1_{\y}$};
    \node at (-0.50,\n-\y+0.5) {\c};
  };
  \foreach \c [count=\n from -1] in {#4}{};
  \foreach \c [count=\x from  0] in {#4}
  {
    \node at (\n-\x+1.5,\n+2.25) {$#3_{\x}$};
    \node at (\x+0.5,\n+1.5) {\c};
  }
}

\newcommand{\drawArea}[5]{%
  \foreach \c [count=\n from -1] in {#1}{};
  \foreach \c [count=\y from  0] in {#1}
  {
    \foreach \x in {1,...,\c}
    {
      \draw[#3, fill] (\n-\y-\x+2+#4,\n-\y+1+#5) circle (#2 pt);
    }
  };
}

\title[Classification of 2-regular simple modules for Nakayama algebras]{A combinatorial classification of \\ 2-regular simple modules for Nakayama algebras}


\author[R.~Marczinzik]{Ren{\'e} Marczinzik}%
\address[R.~Marczinzik]{Institute of algebra and number theory, University of Stuttgart, Germany}
\email{marczire@mathematik.uni-stuttgart.de}

\author[M.~Rubey]{Martin Rubey}%
\address[M.~Rubey]{Fakult\"at f\"ur Mathematik und Geoinformation, TU Wien,
  Austria}%
\email{martin.rubey@tuwien.ac.at}%
\thanks{M.R. was supported by the Austrian Science Fund (FWF): P 29275}

\author[C.~Stump]{Christian Stump}
\address[C.~Stump]{Fakult\"at f\"ur Mathematik, Ruhr-Universit\"at Bochum, Germany}
\email{christian.stump@rub.de}
\thanks{\quad C.S. was supported by DFG grants STU 563/2 ``Coxeter-Catalan combinatorics'' and STU 563/4-1 ``Noncrossing phenomena in Algebra and Geometry''.}

\begin{document}

\begin{abstract}
  Enomoto showed for finite dimensional algebras that the classification of exact structures on the category of finitely generated projective modules can be reduced to the classification of $2$-regular simple modules.
  In this article, we give a combinatorial classification of $2$-regular simple modules for Nakayama algebras and we use this classification to answer several natural questions such as when there is a unique exact structure on the category of finitely generated projective modules for Nakayama algebras.
  We also classify $1$-regular simple modules, quasi-hereditary Nakayama algebras and Nakayama algebras of global dimension at most two.
  It turns out that most classes are enumerated by well-known combinatorial sequences, such as Fibonacci, Riordan and Narayana numbers.
  We first obtain interpretations in terms of the Auslander-Reiten quiver of the algebra using homological algebra, and then apply suitable bijections to relate these to combinatorial statistics on Dyck paths.
\end{abstract}

\maketitle

\section{Introduction}

A \Dfn{Nakayama algebra} is a finite-dimensional algebra over a field $\FF$, all whose indecomposable projective and indecomposable injective modules are uniserial.
The aim of this paper is to provide a dictionary between homological properties of Nakayama algebras and their modules, and combinatorial statistics on (possibly periodic) Dyck paths.
Our main results concern $1$- and $2$-regular simple modules.
By a result of Enomoto (\cite[Theorem 3.7]{En}) the classification of $2$-regular simple modules corresponds to the classification of exact structures on the category of finitely generated projective modules. In general the classification of exact structures on the category of finitely generated projective modules for a finite dimensional algebra is a hard problem and it seems there is no solution yet for a large class of algebras. For Nakayama algebras we use a combinatorial model via Dyck paths and explicit knowledge of the beginning of a minimal projective resolution of a simple module to obtain an elementary description of 2-regular simple modules and use this to give a first classification result for 2-regular simple modules for a large class of algebras. Nakayama algebras are one of the most basic classes of algebras in the representation theory of finite dimensional algebras and we hope that our work can be seen as a foundation for more general classification results for exact structures on the category of finitely generated projective modules for larger classes of algebras such as the recently introduced higher Nakayama algebras, see \cite{JK}.
We also mention that 2-regular simple modules can be used to construct Iwanaga-Gorenstein algebras of finite Cohen-Macaulay type, see \cite[Theorem A]{En}, which gives another motivation for the classification of 2-regular simple modules and equivalently exact structures on the category of finitely generated projective modules.
Several natural questions arise, such as:
\begin{enumerate}
  \item When does the category of finitely generated projective modules of an algebra have a unique exact structure?
  \item How many exact structures on the category of finitely generated projective modules can an algebra in a given class of algebras have at most?
\end{enumerate}
In this article we give a combinatorial description of $2$-regular simple modules for Nakayama algebras and use this to completely answer these two algebraic questions.
In addition, our results also exhibit an interesting interplay between the representation theory and homological algebra of Nakayama algebras on the one hand and combinatorial properties of Dyck paths on the other hand.

\medskip

Let~$Q$ be a finite quiver with path algebra $\FF Q$, and let~$J$ denote the ideal generated by the arrows in~$Q$.
Then a two sided ideal~$I$ is called \Dfn{admissible} if $J^m \subseteq I \subseteq J^2$ for some $m \geq 2$.
In this article we assume that all Nakayama algebras are given by a connected quiver and admissible relations.
Note that this assumption is no loss of generality for algebraically closed fields since every algebra is Morita equivalent to a quiver algebra in this case and all our notions are invariant under Morita equivalence.
Using this language, Nakayama algebras are precisely the algebras $\mathbb{F}Q/I$, such that~$I$ is admissible and~$Q$ is either a linear quiver
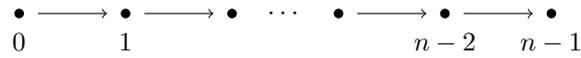
\[
\begin{tikzpicture}[scale=0.7]
  \foreach \pos\lab in {0/0, 2/1, 4/, 6/, 8/n-2, 10/n-1}{
    \coordinate (A\pos) at (\pos,0);
    \draw[fill=black] (A\pos) circle (.08);
    \node[anchor=north] at ($(A\pos)-(0,0.2)$) {$\lab$};
  }
  \foreach \sou\tar in {0/2, 2/4, 6/8, 8/10}{
    \draw[->,shorten <=7pt, shorten >=7pt] (A\sou) -- (A\tar);
  }
  \node at (5,0) {$\cdots$};
\end{tikzpicture}
\]
or a cyclic quiver
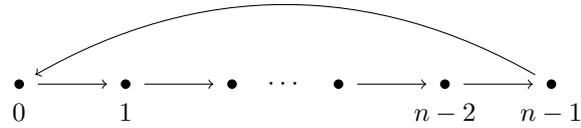
\[
\begin{tikzpicture}[scale=0.7]
  \foreach \pos\lab in {0/0, 2/1, 4/, 6/, 8/n-2, 10/n-1}{
    \coordinate (A\pos) at (\pos,0);
    \draw[fill=black] (A\pos) circle (.08);
    \node[anchor=north] at ($(A\pos)-(0,0.2)$) {$\lab$};
  }
  \foreach \sou\tar in {0/2, 2/4, 6/8, 8/10}{
    \draw[->,shorten <=7pt, shorten >=7pt] (A\sou) -- (A\tar);
  }
  \node at (5,0) {$\cdots$};
  \draw[->,shorten <=7pt, shorten >=7pt] (A10.north) to[out=150,in=30] (A0.north);
\end{tikzpicture}
\]
For textbook introductions to Nakayama algebras we refer for example to~\cite{ARS,AnFul,SkoYam}.
We write \Dfn{LNakayama algebra} for a Nakayama algebra with linear quiver and \Dfn{CNakayama algebra} for a Nakayama algebra with cyclic quiver.
We moreover write \Dfn{$n$-Nakayama algebra}, \Dfn{$n$-LNakayama}, and \Dfn{$n$-CNakayama} in the cases that the respective Nakayama algebra has~$n$ simple modules $S_0,\ldots,S_{n-1}$.
These are in one-to-one correspondence with the vertices of the quiver.

\medskip

In \Cref{sec:prelim} we provide identifications between $(n+1)$-LNakayama algebras and Dyck paths of semilength~$n$ (\Cref{prop:LNak2Dyck}) and between~$n$-CNakayama algebras and certain~$n$-periodic Dyck paths (\Cref{prop:CNak2Dyck}).

\Cref{sec:2-reg} contains the main results of this paper.
These are descriptions of $1$- and $2$-regular simple modules for Nakayama algebras in terms of classical Dyck path statistics (\Cref{thm:1-rises-2-hills} for LNakayama algebras and \Cref{thm:LKC} for CNakayama algebras).
In \Cref{sec:globaldim}, we classify simple modules in Nakayama algebras of global dimension at most two
  (\Cref{gldim2chara}) and Nakayama algebras of global dimension at most two that satisfy the restricted Gorenstein condition
  (\Cref{thm:resGor}).
As corollaries of these classification results, we also obtain explicit enumeration formulas in all considered situations as summarized in \Cref{tab:enum}.
\begin{figure}[t]
  \small
  \resizebox{\textwidth}{!}{
  \begin{tabular}[t]{lll}
    \toprule
    Restriction & Statement\\
    \midrule
    no $1$-regular simples & \cref{cor:no1reg} & (Riordan numbers) \\
    no $2$-regular simples & \cref{cor:no2reg} & (Dyck paths without $2$-hills) \\
    $k$ $1$-regular and $\ell$ $2$-regular simples & \cref{cor:k1regl2reg}\\
    $\ell$ simples of projective dimension $1$ & \cref{cor:pdim1} & (Narayana numbers) \\
    $\ell$ simples of projective dimension $2$ & \cref{cor:pdim2} & (Dyck paths with $\ell$ big returns) \\
    $k$ simples of projective dimension $1$ \\
    \hfil and $\ell$ simples of projective dimension $2$ & \cref{cor:pdim1pdim2GF} \\
    global dimension $2$ and $\ell$ simples of projective dimension $2$ & \cref{gldim2chara} & (subsets of cardinality $2\ell$) \\
    global dimension $2$ and restricted Gorenstein & \cref{cor:resGor} & (Fibonacci numbers) \\
    \midrule
    quasi-hereditary & \cref{cor:quasihercount} & (balanced non-constant binary necklaces) \\
    quasi-hereditary with a simple of dimension~$2$ & \cref{prop:quasiherminentrycount} \\
    quasi-hereditary without $1$-regular simples & \cref{cor:Cno1reg} & (periodic Dyck paths without $1$-rises) \\
    quasi-hereditary without $2$-regular simples & \cref{cor:Cno2reg} & (periodic Dyck paths without $2$-hills) \\
    quasi-hereditary with $\ell$ simples of projective dimension $1$ & \cref{cor:Cpdim1} & (periodic Dyck paths with $\ell$ peaks)\\
    quasi-hereditary with $\ell$ simples of projective dimension $2$ & \cref{cor:Cpdim2} & (periodic Dyck paths with $\ell$ big returns) \\
    global dimension $2$ and $\ell$ simples of projective dimension $2$ & \cref{gldim2chara} & (subsets of cardinality $2\ell$ up to rotation by pairs)\\
    global dimension $2$ and restricted Gorenstein & \cref{cor:resGorCNakayama} & (cyclic compositions of non-singleton parts) \\
    \bottomrule
  \end{tabular}
  }
  \caption{\label{tab:enum}Enumerative results for LNakayama algebras and for CNakayama algebras.}
\end{figure}

\medskip
The translation between Nakayama algebras and Dyck paths made it possible to search
\begin{itemize}
  \item the Online Encyclopedia of Integer Sequences~\cite{OEIS} for counting formulas for the homological properties, and
  \item the combinatorial statistic finder FindStat~\cite{FindStat} for combinatorial interpretations.
\end{itemize}
All major results, including the bijections involved, are based on these searches.
In particular, results found by FindStat suggested the main bijection employed, which is a variant of the Billey-Jockusch-Stanley bijection and the Lalanne-Kreweras involution.

\medskip

For the reader's convenience, we reference integer sequences in this article to the Online Encyclopedia of Integer Sequences~\cite{OEIS} and combinatorial bijections and statistics to FindStat~\cite{FindStat}.
We also provide all discussed homological properties for several small Nakayama algebras in \Cref{fig:Nakayama} for later reference.
Experiments were carried out using the GAP-package QPA~\cite{QPA} and SageMath~\cite{Sage}.

\begin{figure}[ht]
  \centering
  \resizebox{\textwidth}{!}{
    \begin{tabular}[t]{llllll}
      \toprule
      \small Kupisch series & \small $1$-reg. & \small $2$-reg. & \small pdim $1$ & \small pdim $2$ & \small gdim\\
      \midrule
      {}[1]&-&-&-&-&0\\
      \addlinespace
      {}[2, 1]&0&-&0&-&1\\
      \addlinespace
      {}[2, 2, 1]&-&0&1&0&2\\
      {}[3, 2, 1]&0,1&-&0,1&-&1\\
      \addlinespace
      {}[2, 2, 2, 1]&-&-&2&1&3\\
      {}[3, 2, 2, 1]&0&1&0,2&1&2\\
      {}[2, 3, 2, 1]&2&0&1,2&0&2\\
      {}[3, 3, 2, 1]&1&-&1,2&0&2\\
      {}[4, 3, 2, 1]&0,1,2&-&0,1,2&-&1\\
      \addlinespace
      {}[2, 2, 2, 2, 1]&-&-&3&2&4\\
      {}[3, 2, 2, 2, 1]&0&-&0,3&2&3\\
      {}[2, 3, 2, 2, 1]&-&0,2&1,3&0,2&2\\
      {}[3, 3, 2, 2, 1]&1&-&1,3&2&3\\
      {}[4, 3, 2, 2, 1]&0,1&2&0,1,3&2&2\\
      {}[2, 2, 3, 2, 1]&3&-&2,3&1&3\\
      {}[3, 2, 3, 2, 1]&0,3&1&0,2,3&1&2\\
      {}[2, 3, 3, 2, 1]&2&-&2,3&1&3\\
      {}[3, 3, 3, 2, 1]&-&-&2,3&1&3\\
      {}[4, 3, 3, 2, 1]&0,2&-&0,2,3&1&2\\
      {}[2, 4, 3, 2, 1]&2,3&0&1,2,3&0&2\\
      {}[3, 4, 3, 2, 1]&1,3&-&1,2,3&0&2\\
      {}[4, 4, 3, 2, 1]&1,2&-&1,2,3&0&2\\
      {}[5, 4, 3, 2, 1]&0,1,2,3&-&0,1,2,3&-&1\\
      \bottomrule
    \end{tabular}
    \quad
    \begin{tabular}[t]{llllll}
      \toprule
      \small Kupisch series & \small $1$-reg. & \small $2$-reg. & \small pdim $1$ & \small pdim $2$ & \small gdim\\
      \midrule
      {}[3, 2]&-&1&0&1&2\\
      \addlinespace
      {}[2, 3, 2]&-&-&1&0&3\\
      {}[4, 3, 2]&1&2&0,1&2&2\\
      {}[5, 4, 3]&0&-&0,1&2&2\\
      \addlinespace
      {}[2, 2, 3, 2]&-&-&2&1&4\\
      {}[2, 4, 3, 2]&2&-&1,2&0&3\\
      {}[3, 2, 3, 2]&-&1,3&0,2&1,3&2\\
      {}[3, 4, 3, 2]&1&-&1,2&0&3\\
      {}[4, 3, 3, 2]&2&-&0,2&3&3\\
      {}[5, 4, 3, 2]&1,2&3&0,1,2&3&2\\
      {}[3, 5, 4, 3]&-&-&1,2&0&3\\
      {}[6, 5, 4, 3]&0,2&-&0,1,2&3&2\\
      {}[7, 6, 5, 4]&0,1&-&0,1,2&3&2\\
      \addlinespace
      {}[2, 2, 2, 3, 2]&-&-&3&2&5\\
      {}[2, 2, 4, 3, 2]&3&-&2,3&1&4\\
      {}[2, 3, 2, 3, 2]&-&2&1,3&0,2&3\\
      {}[2, 3, 4, 3, 2]&2&-&2,3&1&4\\
      {}[2, 4, 3, 3, 2]&3&-&1,3&0&4\\
      {}[2, 5, 4, 3, 2]&2,3&-&1,2,3&0&3\\
      {}[3, 2, 3, 3, 2]&3&-&0,3&4&4\\
      {}[3, 2, 4, 3, 2]&3&1,4&0,2,3&1,4&2\\
      {}[3, 3, 4, 3, 2]&-&-&2,3&1&4\\
      {}[3, 5, 4, 3, 2]&1,3&-&1,2,3&0&3\\
      {}[4, 3, 3, 3, 2]&-&-&0,3&4&4\\
      {}[4, 3, 4, 3, 2]&2&4&0,2,3&1,4&2\\
      {}[4, 5, 4, 3, 2]&1,2&-&1,2,3&0&3\\
      {}[5, 4, 3, 3, 2]&1,3&-&0,1,3&4&3\\
      {}[5, 4, 4, 3, 2]&2,3&-&0,2,3&4&3\\
      {}[6, 5, 4, 3, 2]&1,2,3&4&0,1,2,3&4&2\\
      {}[3, 3, 5, 4, 3]&-&-&2,3&1&3\\
      {}[3, 6, 5, 4, 3]&3&-&1,2,3&0&3\\
      {}[4, 3, 5, 4, 3]&0,2&-&0,2,3&1&3\\
      {}[4, 6, 5, 4, 3]&2&-&1,2,3&0&3\\
      {}[6, 5, 4, 4, 3]&3&-&0,1,3&4&3\\
      {}[7, 6, 5, 4, 3]&0,2,3&-&0,1,2,3&4&2\\
      {}[4, 7, 6, 5, 4]&1&-&1,2,3&0&3\\
      {}[8, 7, 6, 5, 4]&0,1,3&-&0,1,2,3&4&2\\
      {}[9, 8, 7, 6, 5]&0,1,2&-&0,1,2,3&4&2\\
      \bottomrule
    \end{tabular}
  }
  \caption{\label{fig:Nakayama}Some properties of small LNakayama algebras (left) and of small quasi-hereditary CNakayama algebras (right)}
\end{figure}

\section{Preliminaries}
\label{sec:prelim}

Let~$A$ be an~$n$-Nakayama algebra and let~$e_i$ denote the idempotent corresponding to vertex~$i$ in the corresponding quiver.
The \Dfn{Kupisch series} of~$A$ is the sequence $[c_0,c_1,\dots,c_{n-1}]$, where $c_i \geq 1$ denotes the vector space dimension of the indecomposable projective module $e_iA$.
This sequence uniquely determines the algebra up to isomorphism, see for example~\cite[Theorem 32.9]{AnFul}.
For~$n$-CNakayama algebras we extend the Kupisch series cyclically via $c_i = c_j$ for $i,j \in \ZZ$ with $i \equiv j \mod n$.
Two Kupisch series give isomorphic CNakayama algebras if and only if they coincide up to cyclic rotation, corresponding to the cyclic rotation of the vertices of the quiver.
\medskip

The following identification of Nakayama algebras and Kupisch series is classical and can be found, for example, in~\cite[Chapter~32]{AnFul}.
We repeat it here for the convenience of the more combinatorially inclined reader and to fix notation.
First observe that the Kupisch series $[c_0,\dots,c_{n-1}]$ of an~$n$-Nakayama algebra~$A$ satisfies
\begin{equation}
\label{eq:relative}
  \begin{aligned}
    c_{i+1}+1 \geq c_i&\text{ for all } 0 \leq i < n, \\
          c_i \geq 2\hspace*{2pt} &\text{ for all } 0 \leq i < n-1,
  \end{aligned}
\end{equation}
with indices considered cyclically.
Moreover,~$A$ is an LNakayama algebra if and only if
\begin{equation}
\label{eq:Lboundary}
  c_{n-1}=1.
\end{equation}
A module over a quiver algebra has vector space dimension~$1$ if and only if it is simple, so the latter means that the projective module $e_{n-1}A$ is simple.  Equivalently, the vertex $n-1$ in the quiver has no outgoing arrow.
Together with~\Cref{eq:relative} this forces $c_{n-2} = 2$ for LNakayama algebras.
Otherwise, i.e., if
\begin{equation}
  \label{eq:Cboundary}
  c_{n-1} \geq 2,
\end{equation}
the Nakayama algebra~$A$ is a CNakayama algebra.
Note that~\Cref{eq:relative} forces $c_{n-1} \leq c_0 + 1$ in this case.

\medskip

In total we obtain the following identification.
Here and below, we use the term \Dfn{necklace of length~$n$} for a sequence $[a_0,\dots,a_{n-1}]$ of length~$n$ up to cyclic rotation and write $[a_0,\dots,a_{n-1}]_\circlearrowright$ in this case.

\begin{proposition}
\label{prop:KupischBijection}
  Sending an~$n$-Nakayama algebra to its Kupisch series is a bijection between~$n$-Nakayama algebras and \emph{necklaces} of length~$n$ satisfying~\Cref{eq:relative}.
  It moreover restricts to bijections between
  \begin{enumerate}
    \item\label{it:LKupisch}~$n$-LNakayama algebras and \emph{sequences} of length~$n$ satisfying~\Cref{eq:relative,eq:Lboundary}, and between
    \item\label{it:CKupisch}~$n$-CNakayama algebras and \emph{necklaces} of length~$n$ satisfying~\Cref{eq:relative,eq:Cboundary}.
  \end{enumerate}
\end{proposition}

\begin{remark}
  It is well known that a Nakayama algebra is selfinjective if and only if it is a CNakayama algebra with constant Kupisch series, see for example~\cite[Theorem 6.15 (Chapter IV)]{SkoYam}.
  Over a selfinjective algebra every module is either projective or of infinite projective dimension.
\end{remark}

Let~$A$ be an~$n$-Nakayama algebra with Kupisch series $[c_0,\dots,c_{n-1}]$.
The \Dfn{coKupisch series} is the sequence $[\coc_0, \dots, \coc_{n-1}]$, where $\coc_{i}$ is the vector space dimension of the indecomposable injective module~$D(Ae_i)$ where $D\defeq\Hom_{\FF}(-, \FF)$ denotes the standard duality of a finite-dimensional algebra.
Equivalently, $\coc_i$ is the vector space dimension of the indecomposable projective left module $Ae_i$.
For~$n$-CNakayama algebras we extend the coKupisch series cyclically such that $d_i = d_j$ for $i,j\in\ZZ$ with $i \equiv j$ modulo~$n$ .

The Kupisch and coKupisch series are related by
\begin{equation}
\label{eq:KupischCoKupisch}
  \coc_{i}=\min\big\{k \mid k \geq c_{i-k} \big\},
\end{equation}
see~\cite[Theorem~2.2]{Ful}.
In particular, this implies $\{ c_0,\dots,c_{n-1} \} = \{ \coc_0,\dots,\coc_{n-1} \}$ as multisets.
A sequence is the coKupisch series of an~$n$-Nakayama algebra if and only if the reverse sequence is a Kupisch series.
Let~$A$ and $B$ be~$n$-Nakayama algebras such that the Kupisch series of~$A$ coincides with the reversed coKupisch series of~$B$.
Then also the coKupisch series of~$A$ coincides with the reversed Kupisch series of~$B$.
In particular, interchanging the Kupisch and the reversed coKupisch series is an involution on~$n$-Nakayama algebras.
It is given by mapping an~$n$-Nakayama algebra to its opposite algebra.

\subsection{Nakayama algebras and Dyck paths}
\label{sec:Nak2Dyck}

The \Dfn{Auslander-Reiten quiver} of a repre\-sentation-finite quiver algebra is the quiver with vertices corresponding to the indecomposable modules of the algebra and arrows correspond to the irreducible maps between the indecomposable modules.
We refer for example to~\cite[Chapter~III]{SkoYam} for a detailed introduction to Auslander-Reiten theory.

Nakayama algebras are representation-finite and it is well known that every indecomposable module of an~$n$-Nakayama algebra~$A$ with Kupisch series $[c_0,\dots,c_{n-1}]$ is given up to isomorphism by $\ind_{i,k} \defeq e_i A/e_i J^k$ where~$J$ denotes the Jacobson radical, $i \in \{0,1,\dots,n-1 \}$ and $k \in \{1,2,\dots,c_i \}$.
Note that $\dim \ind_{i,k} = k$.
We identify $\ind_{i, c_i}$ with $e_i A$, which are exactly the indecomposable projective modules, and $\ind_{i+1-d_i, d_i}$ with $D(A e_i)$, which are exactly the indecomposable injective modules.
The modules $S_i = \ind_{i, 1}$ are exactly the simple modules. For~$n$-CNakayama algebras we regard the indices~$i$ of the modules $\ind_{i,k}$ and $S_i$ modulo~$n$, so that they are defined for all $i \in \mathbb{Z}$.

The Auslander-Reiten quiver of an~$n$-Nakayama algebra with Kupisch series given by $[c_0,\dots,c_{n-1}]$ has vertices $\ind_{i,k}$ with $0 \leq i < n$ and $1 \leq k \leq c_i$ and all possible arrows of the form
\[
  \begin{tikzpicture}
    \ARquiver{
        1/0/0/\hspace*{45pt}{\begin{array}{c}\ind_{i,k-1} \\ (i{,}\ i+k-2) \end{array}},
        0/1/0/{\begin{array}{c}\ind_{i,k} \\ (i{,}\ i+k-1) \end{array}}\hspace*{15pt},
        1/2/0/\hspace*{45pt}{\begin{array}{c}\ind_{i-1,k+1} \\ (i-1{,}\ i+k-1) \end{array}}
    }
  \end{tikzpicture}
\]
see, for example,~\cite[Theorem~8.7 (Chapter~III)]{SkoYam}.  Note that exactly the maps $\ind_{i,k}\to\ind_{i-1,k+1}$ are injective, and exactly the maps $\ind_{i,k}\to\ind_{i,k-1}$ are surjective.

\begin{proposition}
\label{indinjnak}
  Let~$A$ be a Nakayama algebra with Kupisch series $[c_0,\dots,c_{n-1}]$.
  The indecomposable module $\ind_{i,m}$ is injective if and only if $c_{i-1} \leq m$.
  In particular, $\ind_{i,c_i}$ is injective if and only if $c_{i-1} \leq c_i$ and dually $D(Ae_i)$ is projective if and only if $d_{i} \geq d_{i+1}$.
\end{proposition}

\begin{proof}
See for example~\cite[Theorem~32.6]{AnFul}.
\end{proof}

We denote by $\tau(\ind_{i,k}) \defeq \ind_{i+1,k}$ the \Dfn{Auslander-Reiten translate} of a non-projective indecomposable module~$\ind_{i,k}$.
In particular, $\tau(S_r) = S_{r+1}$ for non-projective $S_r$, see~\cite[Proposition~2.11 (Chapter~IV)]{ARS}.

As usual, we draw the Auslander-Reiten quiver such that all arrows go from left to right diagonally up or down.
To refer to indecomposable modules in the Auslander-Reiten quiver of a Nakayama algebra it will be convenient to define the \Dfn{coordinates} of $\ind_{i,j}$ to be $(i,i+j-1)$.
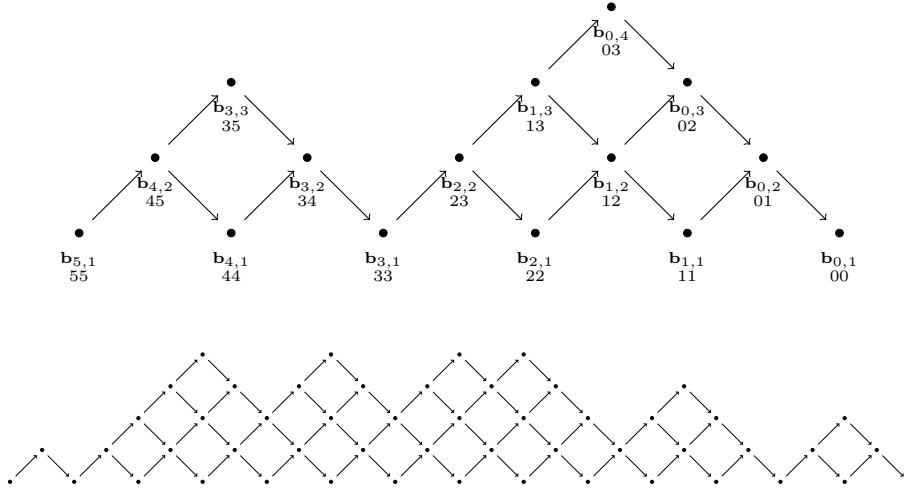
\begin{figure}
  \[
  \begin{tikzpicture}
    \def\textcolor#1{}
    \ARquiver{0/0/1/{\begin{array}{c} \textcolor{red}{\ind_{5,1}} \\ 55 \end{array}},
              2/0/0/{\begin{array}{c} \ind_{4,1} \\ 44 \end{array}},
              4/0/0/{\begin{array}{c} \ind_{3,1} \\ 33 \end{array}},
              6/0/0/{\begin{array}{c} \ind_{2,1} \\ 22 \end{array}},
              8/0/0/{\begin{array}{c} \ind_{1,1} \\ 11 \end{array}},
             10/0/2/{\begin{array}{c} \textcolor{blue}{\ind_{0,1}} \\ 00 \end{array}},
              1/1/1/{\begin{array}{c} \textcolor{red}{\ind_{4,2}} \\ 45 \end{array}},
              3/1/2/{\begin{array}{c} \textcolor{blue}{\ind_{3,2}} \\ 34 \end{array}},
              5/1/1/{\begin{array}{c} \textcolor{red}{\ind_{2,2}} \\ 23 \end{array}},
              7/1/0/{\begin{array}{c} \ind_{1,2} \\ 12 \end{array}},
              9/1/2/{\begin{array}{c} \textcolor{blue}{\ind_{0,2}} \\ 01 \end{array}},
              2/2/3/{\begin{array}{c} \textcolor{green}{\ind_{3,3}} \\ 35 \end{array}},
              6/2/1/{\begin{array}{c} \textcolor{red}{\ind_{1,3}} \\ 13 \end{array}},
              8/2/2/{\begin{array}{c} \textcolor{blue}{\ind_{0,3}} \\ 02 \end{array}},
              7/3/3/{\begin{array}{c} \textcolor{green}{\ind_{0,4}} \\ 03 \end{array}}}
  \end{tikzpicture}
  \]
  \medskip
  \[
  \resizebox{.95\textwidth}{!}{
    \begin{tikzpicture}
      \ARquiver{0/0/0/,2/0/0/,4/0/0/,6/0/0/,8/0/0/,10/0/0/,12/0/0/,14/0/0/,16/0/0/,18/0/0/,20/0/0/,22/0/0/,24/0/0/,26/0/0/,28/0/0/,
                            1/1/0/,3/1/0/,5/1/0/,7/1/0/,9/1/0/,11/1/0/,13/1/0/,15/1/0/,17/1/0/,19/1/0/,21/1/0/,23/1/0/,25/1/0/,27/1/0/,
                            4/2/0/,6/2/0/,8/2/0/,10/2/0/,12/2/0/,14/2/0/,16/2/0/,18/2/0/,20/2/0/,22/2/0/,26/2/0/,
                            5/3/0/,7/3/0/,9/3/0/,11/3/0/,13/3/0/,15/3/0/,17/3/0/,21/3/0/,
                            6/4/0/,10/4/0/,14/4/0/,16/4/0/}
    \end{tikzpicture}
  }
  \]
  \caption{\label{fig:ARexample}The Auslander-Reiten quiver of the Nakayama algebras with Kupisch series $[4,3,2,3,2,1]$ and $[3,2,4,3,5,5,4,5,4,5,4,3,2,2,1]$ and with coKupisch series $[1,2,3,4,2,3]$ and $[1,2,3,2,3,4,3,4,5,5,4,5,4,5,2]$. Modules without incoming arrow from the top left are projective, modules without outgoing arrow to the top right are injective.}
\end{figure}

Given a Nakayama algebra with Kupisch series $[c_0,\dots,c_{n-1}]$ and with coKupisch series $[\coc_0,\dots,\coc_{n-1}]$, these coordinates have the property that the number of vertices with $x$-coordinate~$i$ is given by~$c_i$ and the number of vertices with $y$-coordinate~$j$ is given by~$\coc_j$.
\Cref{fig:ARexample} shows two examples.

\subsubsection{LNakayama algebras and Dyck paths}

Sending an LNakayama algebra to the ``top boundary of its Auslander-Reiten quiver defines a bijection between LNakayama algebras and Dyck paths as follows.
We choose a coordinate system for the $\ZZ^2$-grid by having the \Dfn{horizontal step} $(0,1)$ point left and the \Dfn{vertical step} $(1,0)$ point down.
We identify a square in the $\ZZ^2$-grid with its top-left corner coordinates $(i,j)$.
This is, the square with top-left corner $(i,j)$, top-right corner $(i,j-1)$, bottom-left corner $(i+1,j)$ and bottom-right corner $(i+1,j-1)$ is identified with $(i,j)$.

\medskip

A \Dfn{Dyck path} of semilength~$n$ is a path from $(0,0)$ to $(n,n)$ consisting of vertical and horizontal steps that never goes below the main diagonal $x=y$.
Denote by $\Dyck_n$ the collection of all Dyck paths of semilength~$n$.
In the following we use two \emph{slightly shifted} variants of the area sequence associated with a Dyck path $D \in \Dyck_n$: the \Dfn{area sequence} $[c_0,c_1,\dots,c_n]$ is obtained by setting $c_k$, for $0 \leq k \leq n$, to the number of lattice points with $x$-coordinate~$k$ in the region enclosed by the path and the main diagonal.
Recall that we have identified a square with its top-left corner.
For example, the area sequence of the Dyck path in \Cref{fig:Dyck} is
\[
  [3,2,4,3,5,5,4,5,4,5,4,3,2,2,1].
\]
Similarly, the \Dfn{coarea sequence} $[\coa_0,\ldots,\coa_n]$ is obtained by setting $\coa_k$, for $0 \leq k \leq n$, to the number of lattice points with $y$-coordinate~$k$ in the region enclosed by the path and the main diagonal.
In the example in \Cref{fig:Dyck}, the coarea sequence is
\[
  [1,2,3,2,3,4,3,4,5,5,4,5,4,5,2].
\]
\begin{figure}[t]
  \centering
  \scalebox{0.8}{
    \begin{tikzpicture}[scale=0.75]
      \draw[dotted] (0, 0) grid (14, 14);
      \draw[dotted] (0,  0) --   (14, 14);
      \draw[->] (16, 15) -- node[above] {$y$} (15,15);
      \draw[->] (16, 15) -- node[right] {$x$} (16,14);
      \drawPath%
      {1, 0, 1, 1, 1, 1, 0, 0, 1, 1, 0, 0, 1, 1, 0, 1, 0, 0, 0, 1, 1, 0, 0, 0, 1, 1, 0, 0}{0}{black}{(0,0)}{2}
      \drawArea{3,2,4,3,5,5,4,5,4,5,4,3,2,2,1}{4}{lightgrey}{0}{0}
      \drawLabel%
      {c}      {3,2,4,3,5,5,4,5,4,5,4,3,2,2,1} 
      {\coa}   {2,5,4,5,4,5,5,4,3,4,3,2,3,2,1} 
    \end{tikzpicture}
  }
  \caption{The Dyck path of semilength $14$ corresponding to the Auslander-Reiten quiver in the bottom example in \Cref{fig:ARexample}.}
  \label{fig:Dyck}
\end{figure}
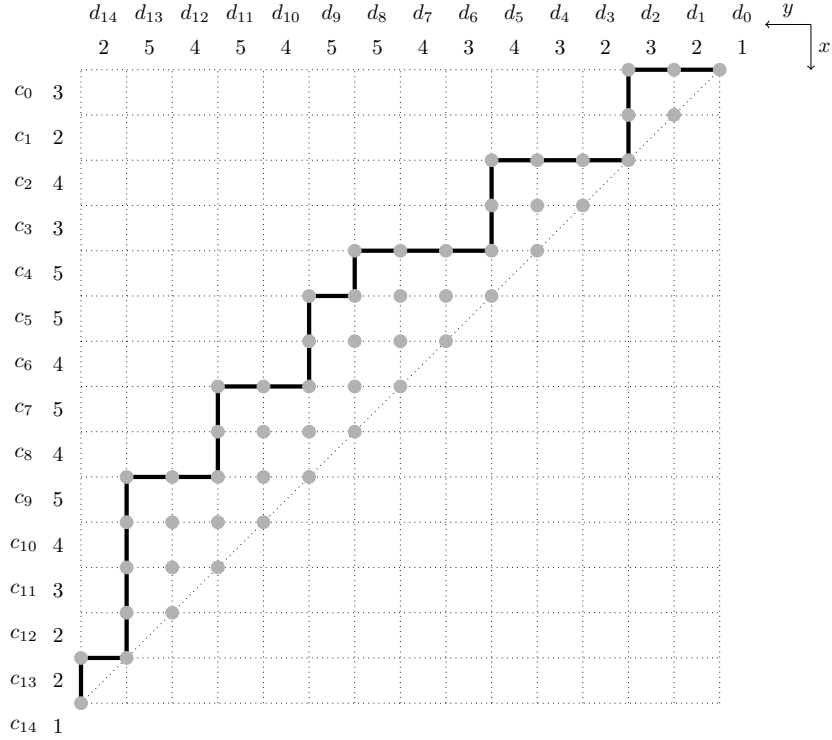
Sending a Dyck path $D \in \Dyck_n$ to its area sequence is obviously a bijection between $\Dyck_n$ and sequences $[c_0,\ldots,c_n]$ satisfying Conditions \eqref{eq:relative} and~\eqref{eq:Lboundary}.
As seen in \Cref{prop:KupischBijection}, these are exactly the same conditions as those for Kupisch series of $(n+1)$-LNakayama algebras.
This observation yields the following formalization of the pictorially indicated bijection between LNakayama algebras and Dyck paths given by sending an algebra to the top boundary of its Auslander-Reiten quiver.

\begin{proposition}
  \label{prop:LNak2Dyck}
  The map sending an $(n+1)$-LNakayama algebra~$A$ to the unique Dyck path~$D$ of semilength~$n$ such that the Kupisch series of~$A$ coincides with the area sequence of~$D$ is a bijection between $(n+1)$-LNakayama algebras and Dyck paths of semilength~$n$.
\end{proposition}

This connection has already appeared in the literature, see for example~\cite[page~256]{Rin} for (a variant of) this bijection.
We moreover observe that \Cref{eq:KupischCoKupisch} implies that the coarea sequence of~$D$ also equals the coKupisch series of~$A$.

\subsubsection{CNakayama algebras and periodic Dyck paths}

Replacing the initial condition in \Cref{eq:Lboundary} for LNakayama algebras with the initial condition in \Cref{eq:Cboundary} for CNakayama algebras we obtain a description of these in terms of periodic Dyck paths.

\medskip

A \Dfn{balanced binary~$n$-necklace} is a binary necklace consisting of~$n$ white and~$n$ black beads.
In the above language, we represent a white bead by the letter~$v$ and the black bead by the letter~$h$, so that a balanced binary~$n$-necklace is a sequence of~$n$ letters~$v$ and~$h$ each, considered up to cyclic rotation.
Formally, an \Dfn{$n$-periodic Dyck path} is a balanced binary~$n$-necklace together with an integer $c \geq 0$; we refer to this integer as its \Dfn{global shift}.
This corresponds to an actual path in the $\ZZ^2$-grid up to diagonal translations together with an explicit choice of a diagonal as follows.
One draws a bi-infinite path given by the the balanced binary~$n$-necklace where white beads represent vertical steps and black beads represent horizontal steps.
This path is chosen so that it stays weakly but not strictly above the diagonal $y = x+c$, and two paths are identified if they coincide up to diagonal translation.

\medskip

The \Dfn{area sequence} of an~$n$-periodic Dyck path is the necklace $[c_0,\dots,c_{n-1}]_\circlearrowright$, where $c_k$ is the number of lattice points with $y$-coordinate~$k$ in the region enclosed by the path and the chosen diagonal.
Note that, in contrast to the area sequence of an ordinary Dyck path of semilength~$n$, the area sequence of an~$n$-periodic Dyck path has length~$n$ rather than $n+1$.
The \Dfn{coarea sequence} is defined accordingly.
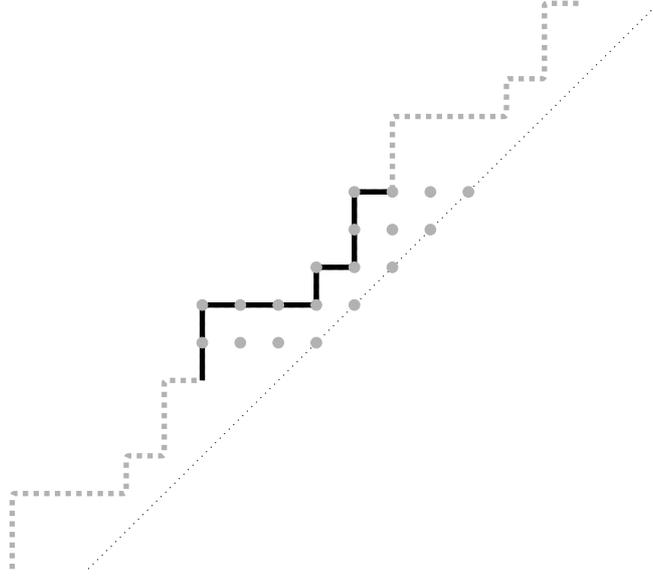
\begin{figure}
  \centering
  \begin{tikzpicture}[scale=0.5]
    \draw[dotted] (-3,-5) -- (12,10);

    \drawPath%
    {1,1,0,0,0,1,0,1,1,0,1,1,0,0,0,1,0,1,1,0,1,1,0,0,0,1,0,1,1,0}{0}{lightgrey, dotted}{(-5,-5)}{2}

    \drawPath%
    {1,1,0,0,0,1,0,1,1,0}{0}{black}{(0,0)}{2}
    \drawArea{4,3,3,5,4}{4}{lightgrey}{3}{1}

  \end{tikzpicture}
  \caption{\label{fig:Dyckrotation}A $5$-periodic Dyck path of global shift~$1$.}
\end{figure}
\Cref{fig:Dyckrotation} shows the $5$-periodic Dyck path $[h,v,v,h,v,h,h,h,v,v]_\circlearrowright$ with global shift~$1$ and area sequence $[4,3,3,5,4]_\circlearrowright$.

\medskip

Similar to the case of ordinary Dyck paths, it is immediate from the definition that sending an~$n$-periodic Dyck path to its area sequence (or, respectively, its reversed coarea sequence) is a bijection between~$n$-periodic Dyck paths and necklaces $[c_0,\ldots,c_{n-1}]_\circlearrowright$ satisfying Conditions \eqref{eq:relative} and~\eqref{eq:Cboundary}.
As seen in \Cref{prop:KupischBijection}, these are exactly the same conditions as those for Kupisch series of $n$-CNakayama algebras.
This observation yields the following proposition.

\begin{proposition}
\label{prop:CNak2Dyck}
  Fix $c \geq 0$.
  Sending an~$n$-CNakayama algebra with Kupisch series $[c_0,\ldots,c_{n-1}]_\circlearrowright$ to the~$n$-periodic Dyck path with area sequence $[c_0,\dots,c_{n-1}]_\circlearrowright$ is a bijection between~$n$-CNakayama algebras whose Kupisch series have minimal entry $c+2$ and~$n$-periodic Dyck paths of global shift~$c$.
\end{proposition}

This proposition also has the following corollary.

\begin{corollary}[{\cite[Exercise~3.1.10c]{BLL}}]\label{cor:Cmin}
  For any $c \geq 0$, the number of~$n$-CNakayama algebras whose
  Kupisch series has minimal entry~$c+2$ equals the number of
  balanced binary~$n$-necklaces\footnote{\OEIS{A003239}}.
  Explicitly, this number is
  \[
    \frac{1}{2n}\sum_{k\divides n}\phi\big(n/k\big)\binom{2k}{k},
  \]
  where~$\phi$ is Euler's totient, the number of integers relatively prime to the argument.
\end{corollary}

\subsection{Combinatorial statistics on Dyck paths}
\label{sec:Dyck-statistics}

It will be convenient to give names to certain special points in (periodic) Dyck paths.
Note that, a priori, we cannot refer to individual steps in periodic Dyck paths or elements of the associated necklace, because they are only defined up to rotation.
However, we can fix a (cyclic) labelling of the coordinates with $0,\dots, n-1$ as provided by the correspondence with the simple modules of the associated~$n$-CNakayama algebra.

\begin{definition}
\label{def:Dyck-statistics}
  Let~$D$ be a Dyck path of semilength~$n$, or an $n$-periodic Dyck path.

  A \Dfn{peak}\footnote{\FindStat{St000015}} at coordinates $(i,j)$ is a horizontal step with $x$-coordinate~$i$ followed by a vertical step with $y$-coordinate~$j$.
  A point $(i,j)$ is a peak if and only if $c_i \geq c_{i-1}$ and $j=i+c_i-1$, except that for Dyck paths $(0, c_0-1)$ is also a peak.

  A \Dfn{valley}\footnote{\FindStat{St000053}} at coordinates $(i,j)$ is a vertical step with $y$-coordinate~$j$ followed by a horizontal step with $x$-coordinate~$i$.
  A point $(i,j)$ is a valley if and only if $c_i \geq c_{i-1}$ and $j=i+c_{i-1}-2$.

  A \Dfn{$1$-rise}\footnote{\FindStat{St000445}} at coordinates $(i,j)$ is a horizontal step with $x$-coordinate~$i$ and final $y$-coordinate $j$, which is neither preceded nor followed by a horizontal step.
  A point $(i,j)$ is a $1$-rise if and only if $c_i = c_{i-1}$ for $i>0$, or $c_0=2$ for $i=0$, and $j=i+c_i-1$.

  A \Dfn{double rise} at coordinates~$(i,j)$ is a segment of two consecutive horizontal steps whose midpoint has coordinates~$(i,j)$.
  There is a double rise with midpoint at $y=j$ if and only if $d_j > 1$ and  $d_{j+1}-d_j=1$.

  A \Dfn{double fall} at coordinates~$(i,j)$ is a segment of two consecutive vertical steps whose midpoint has coordinates~$(i,j)$.
  There is a double rise with midpoint at $x=i$ if and only if $c_i > 1$ and $c_{i-1}-c_i=1$.

  A \Dfn{return}\footnote{\FindStat{St000011}} at position~$i$ is a (necessarily vertical) step with final coordinates $(i,i)$.  There is a return at position~$i$ if and only if $c_{i-1}=2$ and $i>0$, and if and only if $d_{i+1}=2$ or, for Dyck paths, $i=n$.  A Dyck path is \Dfn{prime} if it has only one return.

  A \Dfn{$1$-cut} at position~$i$ is an occurrence of a horizontal step with $x$-coordinate~$i$ and a vertical step with $y$-coordinate $i+1$.

  A \Dfn{$k$-hill}\footnote{\FindStat{St000674}, \FindStat{St001139}, \FindStat{St001141}} at position~$i$ is a segment of~$k$ consecutive horizontal steps followed by~$k$ consecutive vertical steps, starting at $(i,i)$.

  A \Dfn{rectangle} at coordinates~$(i+1,j)$ is a valley at $(i+1,j)$, such that the next valley has $x$-coordinate strictly larger than $j+1$.
  In terms of area sequences, this is $c_{i+1}+1 = c_i + c_{i+c_i}$, with $j=i+c_i-1$.
\end{definition}

\begin{figure}[t]
  \centering
  \begin{tikzpicture}[scale=0.7]
    \draw[->] (11, 5) -- node[above] {$y$} (10,5);
    \draw[->] (11, 5) -- node[right] {$x$} (11,4);
    \draw[dotted] (-1,0) grid (10,4);
    \drawPath%
    {1,0}{0}{black}{(0,1)}{2}
    \drawPath%
    {0,1}{0}{black}{(2,2)}{2}
    \drawPath%
    {1,0,1}{0}{black}{(4,1)}{2}
    \drawPath%
    {0,0}{0}{black}{(6,2)}{2}
    \drawPath%
    {1,1}{0}{black}{(9,1)}{2}
    \draw[fill=black] (0,2) circle (.15);
    \draw[fill=black] (3,2) circle (.15);
    \draw[fill=black] (4,2) circle (.15);
    \draw[fill=black] (7,2) circle (.15);
    \draw[fill=black] (9,2) circle (.15);
  \end{tikzpicture}
  \caption{Coordinates of peaks, valleys, 1-rises, double rises, and double falls.}
  \label{fig:DyckStatistics}
\end{figure}

We refer to \Cref{fig:regular-on-Dyck-paths} on page~\pageref{fig:regular-on-Dyck-paths} for examples of returns, $1$-cuts and $2$-hills, and to \Cref{fig:gdtwoexample} on page~\pageref{fig:gdtwoexample} for examples of rectangles.

\subsection{Some homological properties of finite dimensional algebras}
\label{sec:homprop}

In this section, we recall several known homological properties of finite dimensional algebras that we need in later sections.

\medskip

We quickly recall the definitions of projective dimension and $\Ext$ here and refer for example to \cite{Ben} for detailed information.
Recall that the \Dfn{projective cover} of a module~$M$ is by definition the unique surjective map (up to isomorphism) $P \rightarrow M$ such that $P$ is projective of minimal vector space dimension. Dually the \Dfn{injective envelope} of~$M$ is by definition the unique injective map $M \rightarrow I$ such that $I$ is injective of minimal vector space dimension. One often just speaks of $P$ as the projective cover for short and also as $I$ being the injective envelope. We will often use that a module $M$ is isomorphic to its projective cover $P$ if and only if $M$ has the same vector space dimension as its projective cover $P$. This follows immediately from the fact that a projective cover is a surjection and that a module homomorphism is an isomorphism if and only if it is surjective and both modules have the same vector space dimension.
For a module~$M$, the \Dfn{first syzygy module} $\Omega^{1}(M)$ is by definition the kernel of the projective cover $P \rightarrow M$ of~$M$. Inductively, one then defines for $n \geq 0$ the \Dfn{$n$-th syzygy module} of~$M$ as $\Omega^n(M)\defeq \Omega^1(\Omega^{n-1}(M))$ with $\Omega^0(M)=M$.
The \Dfn{projective dimension} $\pd(M)$ of~$M$ is defined as the smallest integer $n \geq 0$ such that $\Omega^n(M)$ is projective and as infinite in case no such~$n$ exists.
For two~$A$-modules~$M$ and $N$ one defines $\Ext_A^1(M,N)$ as $\Ext_A^1(M,N)\defeq D(\overline{\Hom}_A(N,\tau(M)))$, where $\tau(M)$ denotes the Auslander-Reiten translate of~$M$ and $\overline{\Hom}_A(X,Y)$ denotes the space of homomorphisms between two~$A$-modules $X$ and $Y$ modulo the space of homomorphisms between $X$ and $Y$ that factor over an injective~$A$-module.
For $n \geq 1$, one then defines $\Ext_A^n(M,N)\defeq\Ext_A^1(\Omega^{n-1}(M),N)$.
We furthermore define $\Ext_A^0(M,N)\defeq\Hom_A(M,N)$.
Note that we choose here to present the definition of $\Ext$ in the probably shortest way possible (using the Auslander-Reiten formulas, see for example \cite[Theorem~6.3. (Chapter~III)]{SkoYam}) and we refer for example to \cite[Chapter~2.4]{Ben} for the classical definition.
For the practical calculation of the projective cover, injective envelope and syzygies of modules in Nakayama algebras we refer the reader to the preliminaries of \cite{Mar}.
\begin{lemma}
\label{extformula}
  Let~$A$ be a finite-dimensional algebra.
  Let~$S$ be a simple~$A$-module and~$M$ an~$A$-module with minimal projective resolution
  \[
    \cdots \rightarrow P_i \rightarrow \cdots \rightarrow P_1 \rightarrow P_0 \rightarrow M \rightarrow 0.
  \]
  For $\ell \geq 0$, $\Ext_A^\ell(M,S) \neq 0$ if and only if there is a surjection $P_\ell \rightarrow S$.
  Dually, let
  \[
    0 \rightarrow M \rightarrow I_0 \rightarrow I_1 \rightarrow \cdots \rightarrow I_i \rightarrow \cdots
  \]
  be a minimal injective coresolution of~$M$.
  For $\ell \geq 0$, $\Ext_A^\ell(S,M) \neq 0$ if and only if there is an injection $S \rightarrow I_\ell$.
\end{lemma}

\begin{proof}
  See for example~\cite[Corollary~2.5.4]{Ben}.
\end{proof}

\section{1-regular and 2-regular simple modules}
\label{sec:2-reg}

In this section we provide characterizations of $1$- and $2$-regular simple modules of Nakayama algebras in terms of Kupisch series.
We then exhibit bijections that transform these conditions into local properties of Dyck paths and periodic Dyck paths.

\begin{definition}
\label{regulardefinition}
  Let~$A$ be a finite-dimensional algebra and let~$S$ be a simple~$A$-module.
  For~$k \in \NN$, the module~$S$ is \Dfn{$k$-regular} if
  \begin{enumerate}
    \item $\pd(S)=k$, \label{it:regdef1}
    \item $\Ext_A^i(S,A)=0$ for $0 \leq i < k$, and
    \item $\dim\Ext_A^k(S,A)=1$.
  \end{enumerate}
\end{definition}
Recall that the condition $\dim\Ext_A^k(S,A)=1$ is equivalent to the left~$A$-module $\Ext_A^k(S,A)$ being simple, since modules over quiver algebras are simple if and only if they have vector space dimension equal to one.

\medskip

The definition of $k$-regular simple modules is motivated by the notion of the restricted Gorenstein condition which is used in higher Auslander-Reiten theory, see for example~\cite[Proposition~1.4 and Theorem~2.7]{Iy}.
We study the restricted Gorenstein condition in the special case of Nakayama algebras of global dimension at most two in \Cref{sec:globaldim}.
The simple module $S_{n-1}$ for an $n$-LNakayama algebra is the unique simple projective module and thus $S_{n-1}$ is never $k$-regular for $k \geq 1$.
Thus it is no loss of generality to exclude the simple module $S_{n-1}$ in our treatment of $k$-regular simple modules.

\medskip

The most important case of $k$-regularity is $2$-regularity, which was recently used by Enomoto~\cite{En} to reduce the classification of exact structures on categories of
finitely generated projective~$A$-modules for Artin algebras~$A$ to the classification of $2$-regular simple modules.
We refer to~\cite{Bue} for the definitions and discussions of exact categories.
Enomoto's result, restricted to finite-dimensional algebras, is as follows.

\begin{theorem}[\protect{\cite[Theorem~3.7]{En}}]
\label{thm:enomoto}
  Let~$A$ be a finite-dimensional algebra and let $\mathcal{E}$ be the category of finitely generated projective~$A$-modules.
  Then there is a bijection between
  \begin{enumerate}
    \item exact structures on $\mathcal{E}$ and
    \item sets of isomorphism classes of $2$-regular simple~$A$-modules.
  \end{enumerate}
\end{theorem}

Thus when an algebra $A$ has exactly $m$ $2$-regular simple modules, it has exactly $2^m$ exact structures on the category of finitely generated projective modules.

\subsection{Description in terms of Kupisch series}

For any~$k$, a $k$-regular simple module~$S_i$ is non-projective by \Cref{regulardefinition}\eqref{it:regdef1}.
In the case of~$n$-LNakayama algebras this means that $i < n-1$, whereas this is no restriction for CNakayama algebras since the latter do not have projective simple modules.
Throughout this section, let~$A$ denote an~$n$-Nakayama algebra with Kupisch series $[c_0,\dots,c_{n-1}]$ and coKupisch series $[d_0,\dots,d_{n-1}]$ and let $S_i$ denote a simple~$A$-module corresponding to the vertex~$i$.
\begin{theorem}
\label{regchara}
  A simple non-projective module $S_i$ is
  \begin{enumerate}
   \item\label{it:1reg} $1$-regular\footnote{\FindStat{St001126}} if and only if
     $c_i - c_{i+1} = d_{i+1} - d_i = 1$,

    \item\label{eq:2regchara} $2$-regular\footnote{\FindStat{St001125}} if and only if
      \[
        c_i = d_{i+2} = 2 \quad \text{and}\quad c_{i+1} - c_{i+2} = d_{i+1} - d_i = 1.
      \]
  \end{enumerate}
\end{theorem}

The first step towards this theorem is a description of the non-projective simple modules of projective dimensions one and two.

\begin{proposition}
\label{smallprojdim}
  A simple non-projective module $S_i$ has
  \begin{enumerate}
  \item\label{it:pd1} $\pd(S_i)=1$\footnote{\FindStat{St001007}} if and only if $c_{i+1} + 1 = c_i$, and
    \item\label{it:pd2} $\pd(S_i)=2$\footnote{\FindStat{St001011}} if and only if $c_{i+1} + 1 = c_{i+c_i} + c_i$.
  \end{enumerate}
\end{proposition}

\begin{proof}
  We have that $\pd(S_i) = 1$ if and only if the module $e_i J$ in the short exact sequence
  \[
    0 \rightarrow e_i J \rightarrow e_i A \rightarrow S_i \rightarrow 0
  \]
  is projective.
  This is the case if and only if $e_i J$ is isomorphic to its projective cover $e_{i+1}A$, which is equivalent to $c_i -c_{i+1} = 1$ by comparing vector space dimensions and using that $\dim(e_{i+1}A)= c_{i+1}$ and $\dim(e_i J^k)=c_i -k$.

  \medskip

  The beginning of a minimal projective resolution of~$S_i$ is given by splicing together the two short exact sequences
  \[
    0 \rightarrow e_i J \rightarrow e_i A \rightarrow S_i \rightarrow 0
  \]
  \[
    0 \rightarrow e_{i+1} J^{c_i-1} \rightarrow e_{i+1} A \rightarrow e_i J \rightarrow 0.
  \]
  We have already seen in~\eqref{it:pd1} that $\pd(S_i) \geq 2$ if and only if $c_i \leq c_{i+1}$.
  Moreover, $\pd(S_i)=2$ if additionally $e_{i+1} J^{c_i-1}$ is projective.
  Now $e_{i+1} J^{c_i-1}$ being projective is equivalent to the condition that it is isomorphic to its projective cover $e_{i+c_i}A$, which in turn translates into the condition $c_{i+1}-(c_i -1)=c_{i+c_i}$ by comparing vector space dimensions.
\end{proof}

\begin{lemma}
  \label{lemmahomandext1}
  For a simple non-projective module $S_i$, we have
  \begin{enumerate}
  \item $\Hom_A(S_i,A)=0 \Leftrightarrow d_{i+1} = d_i + 1,$
  \item $\Ext_A^1(S_i,A)=0 \Leftrightarrow c_i < c_{i+1} + 1.$
  \end{enumerate}
\end{lemma}

\begin{proof}

  Note that $\Hom_A(S_i,A)=0$ if and only if $S_i$ does not appear in the socle of~$A$, which is equivalent to the injective envelope $I(S_i)=D(Ae_i)$ of $S_i$ being non-projective (here we use that the injective envelope of~$A$ is projective-injective for every Nakayama algebra, see for example~\cite[Theorem~32.2]{AnFul}).  This translates into the condition $d_{i+1} > d_i$ by using \Cref{indinjnak}.

  \medskip

  For the second property, note that $\Ext_A^1(S_i,A)=0$ if and only if $\Ext_A^1(S_i,e_r A)=0$ for every indecomposable non-injective module $e_r A$.  Thus, suppose that $e_r A$ is non-injective, then
  $$0 \rightarrow e_r A \rightarrow D(A e_{r+c_r-1}) \rightarrow D(Ae_{r-1})$$
  is the beginning of a minimal injective coresolution of a non-injective $e_rA$, see for example~\cite[Preliminaries]{Mar}.
  \Cref{extformula} entails that $\Ext_A^1(S_i,e_r A)\neq 0$ if and only if there is an injection $S_i\to D(Ae_{r-1})$.  Since $S_i = \ind_{i,1}$ and $D(Ae_{r-1}) = \ind_{r-d_{r-1}, d_{r-1}}$, such an injection exists if and only if $i=r-1$.
\end{proof}

\begin{lemma}
\label{lemmahomandext2}
  We have the following two properties for a simple non-projective module~$S_i$:
 \begin{enumerate}
    \item If $\Hom_A(S_i,A)=0$ and~$\pd(S_i) = 1$, then
    \[
      \dim\Ext_A^1(S_i,A)=1 \Leftrightarrow d_{i+1} = d_i +1.
    \]

    \item If $\Hom_A(S_i,A) = \Ext_A^1(S_i,A) = 0$ and~$\pd(S_i) = 2$, then
    \[
      \dim\Ext_A^2(S_i,A)=1 \Leftrightarrow d_{i+1} + 1 = d_i + d_{i+c_i}.
    \]
  \end{enumerate}
\end{lemma}
\begin{proof}
  For the first property, we apply the left exact functor $\Hom_A(-,A)$ to the short exact sequence
  \[
    0 \rightarrow e_i J \rightarrow e_i A \rightarrow S_i \rightarrow 0
  \]
  and use that $e_iJ \cong e_{i+1}A$ (since~$S_i$ is assumed to have projective dimension equal to one). We obtain the exact sequence
  \[
    0 \rightarrow \Hom_A(S_i,A) \rightarrow \Hom_A(e_iA,A) \rightarrow \Hom_A(e_{i+1}A ,A) \rightarrow \Ext_A^1(S_i,A) \rightarrow 0.
  \]
  Comparing dimensions and using $\Hom_A(S_i,A)=0$, we obtain the condition
  \begin{align*}
    1 &= \dim\Ext_A^1(S_i,A) \\
      &= \dim\Hom_A(S_i,A)+\dim\Hom_A(e_{i+1}A,A)-\dim\Hom_A(e_iA,A) \\
      &= \dim(A e_{i+1})-\dim(A e_i) \\
      &= d_{i+1}-d_i.
  \end{align*}

  For the second property, we apply the left exact functor $\Hom_A(-,A)$ to the short exact sequence $$0 \rightarrow e_i J \rightarrow e_i A \rightarrow S_i \rightarrow 0,$$
  and we obtain the exact sequence
  $$0 \rightarrow \Hom_A(S_i,A) \rightarrow \Hom_A(e_iA,A) \rightarrow \Hom_A(e_i J ,A) \rightarrow \Ext_A^1(S_i,A) \rightarrow 0.$$
  The condition $\Ext_A^1(S_i,A)=0$, together with $\Hom_A(S_i,A)=0$, is equivalent to
  \[
    \Hom_A(e_iA,A) \cong \Hom_A(e_i J ,A),
  \]
  which translates into the condition $\dim\Hom_A(e_i J ,A)=d_i$.
  Now we apply the functor $\Hom_A(-,A)$ to the short exact sequence
  $$0 \rightarrow e_{i+1} J^{c_i-1} \rightarrow e_{i+1} A \rightarrow e_i J \rightarrow 0,$$
  where we use that $e_{i+1} J^{c_i-1} \cong e_{i+c_i}A$ is projective since $S_i$ is assumed to have projective dimension equal to two.
  We obtain the exact sequence
  $$0 \rightarrow \Hom_A(e_i J, A) \rightarrow \Hom_A(e_{i+1}A,A) \rightarrow \Hom_A(e_{i+c_i}A,A) \rightarrow \Ext_A^1(e_iJ,A) \rightarrow 0.$$
  Now note that $\Ext_A^1(e_iJ,A) \cong \Ext_A^1(\Omega^1(S_i),A) \cong \Ext_A^2(S_i,A)$.
  Comparing dimensions we obtain
  \begin{align*}
  1&=\dim\Ext_A^2(S_i,A)\\
  &=\dim\Hom_A(e_i J, A)+\dim\Hom_A(e_{i+c_i}A,A)-\dim\Hom_A(e_{i+1}A,A)\\
  &=d_i+d_{i+c_i}-d_{i+1}. \qedhere
  \end{align*}
\end{proof}

\begin{proof}[Proof of \Cref{regchara}]
  The description of $1$-regular simple modules is a direct consequence of the respective first items in \Cref{smallprojdim} and \Cref{lemmahomandext1,,lemmahomandext2}.
  These lemmas also give that~$S_i$ is $2$-regular if and only if
  \begin{equation}
    \label{eq:2regreduction}
    \begin{aligned}
      c_i       < c_{i+1} + 1 &= c_i + c_{i+c_i}, \\
      d_i+2 = d_{i+1} + 1 &= d_i+d_{i+c_i}.
    \end{aligned}
  \end{equation}

  We simplify these conditions as follows.
  The first condition implies that there are $c_{i+1} + 1 - c_i = c_{i+c_i} > 0$ horizontal steps with $x$-coordinate $i+1$ in the (possibly periodic) Dyck path corresponding to~$A$.
  Thus there is a valley at~$(i+1,i+c_i-1)$.
  The second condition implies $d_{i+c_i} = 2$.
  Therefore, the valley is on the main diagonal, which in turn implies that $c_i=2$.
  Conversely,
   \begin{align*}
      c_i + c_{i+c_i} &= 2 + c_{i+2} = c_{i+1} + 1\\
      d_i + d_{i+c_i} &= d_i + 2 = d_{i+1} + 1. \qedhere
   \end{align*}
\end{proof}

\begin{example}
  Let~$A$ be the $5$-LNakayama algebra with Kupisch series $[4,3,2,2,1]$ and coKupisch series $[1,2,3,4,2]$.
  By \Cref{smallprojdim}, the simple modules $S_0, S_1$ and $S_3$ have projective dimension $1$ and the simple module $S_2$ has projective dimension $2$.  The simple module $S_4$ is projective.
  To see that $S_0$ and $S_1$ are $1$-regular while $S_3$ is not, we compute
  \[
  d_1-d_0= d_2-d_1 = 1 \neq d_4-d_3.
  \]
  Moreover, $S_2$ is $2$-regular because
  \[
    c_2 = d_4 = 2 \quad\text{and}\quad c_3 - c_4 = d_3 - d_2 = 1.
  \]
\end{example}

\subsection{Description in terms of Dyck path statistics for LNakayama algebras}

Based on \Cref{regchara}, we obtain combinatorial reformulations of $1$- and $2$-regularity in terms of Dyck paths.  The first of these is a direct translation into the language of Dyck paths, the second uses a classical involution on Dyck paths to obtain a more local description, and the third uses another bijection which yields a completely local description in terms of two classical statistics.

\begin{theorem}
\label{thm:Dyck-regular}
  Let~$A$ be an~$n$-LNakayama algebra and let~$D$ be its corresponding Dyck path of semilength $n-1$.
  Let~$\widehat D$ be the path obtained from $D$ by adding a horizontal step from $(0,-1)$ to $(0,0)$ and a vertical step from $(n-1,n-1)$ to $(n,n-1)$.
  Then the $A$-module
  \begin{enumerate}
  \item\label{it:Dyck-1-regular} $S_i$ is $1$-regular if and only if $\widehat D$ has a double rise with $y$-coordinate~$i$ and a double fall with $x$-coordinate~$i+1$.
  \item\label{it:Dyck-2-regular} $S_i$ is $2$-regular if and only if $\widehat D$ has a vertical step with final coordinates $(i+1,i+1)$, a double rise with $y$-coordinate~$i$ and a double fall with $x$-coordinate~$i+2$.
  \end{enumerate}
\end{theorem}
\begin{proof}
  This is immediate from \Cref{regchara}.
\end{proof}

\begin{figure}
  \centering
  \resizebox{.6\textwidth}{!}{
  \begin{tikzpicture}[scale=0.75]
    \draw[dotted] (0, 0) grid (14, 14);
    \draw[dotted] (0, 0) -- (14, 14);
    \draw[->] (15, 15) -- node[above] {$y$} (14,15);
    \draw[->] (15, 15) -- node[right] {$x$} (15,14);
\drawPath%
{1, 1, 0, 0, 1, 1, 1, 0, 0, 1, 1, 1, 0, 1, 0, 1, 0, 0, 0, 1, 0, 1, 1, 0, 0, 0, 1, 0} 
{1}{black}{(0,0)}{1}
\drawPath%
{1, 0, 1, 1, 0, 0, 1, 1, 1, 1, 0, 1, 0, 0, 1, 1, 1, 0, 1, 0, 0, 0, 0, 0, 1, 1, 0, 0} 
{0}{blue}{(0,0)}{2}
\drawLK{0/1/1/2,2/3/3/5,2/4/7/10,4/6/8/10,4/7/11/13,10/12/12/13} 
\drawPeaksValleys%
{1/1,3/2,7/4,8/5,11/7,12/8,14/13,2/3,4/6,5/9,6/10,9/11,10/12,13/14} 
  \end{tikzpicture}
  }
  \caption{A mirrored Dyck path~$D$ (in black, thin), its Lalanne-Kreweras
    involution (in blue) and the permutation (indicated by
    crosses) of the $321$-avoiding permutation obtained by applying
    the Billey-Jockusch-Stanley bijection.}
  \label{fig:Lalanne-Kreweras}
\end{figure}
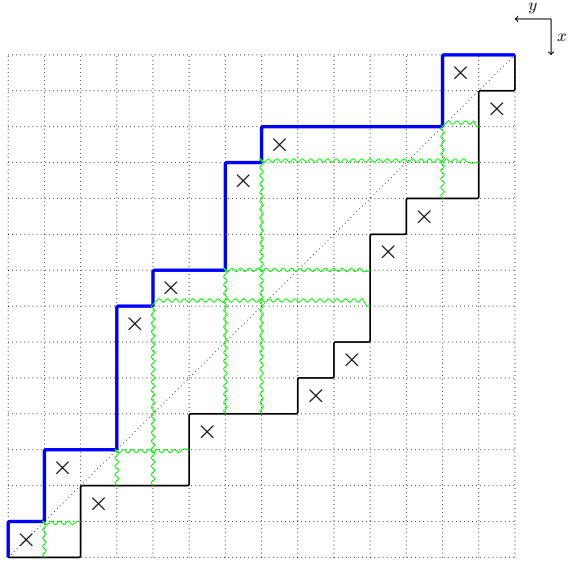

As announced, the second reformulation uses a classical involution which we now recall.
\begin{definition}[\cite{Lalanne, Kreweras}]\label{dfn:LK}
  The \Dfn{Lalanne-Kreweras involution}\footnote{\FindStat{Mp00120}} $\LK$ on Dyck paths of
  semilength~$n$ is the following map:
  \begin{enumerate}
  \item Mirror the Dyck path~$D$ to obtain a path below the main diagonal, from the top
    right to the bottom left.
  \item Draw a vertical line emanating from the midpoint of each
    double horizontal step and a horizontal line emanating from the
    midpoint of each double vertical step.
  \item Mark the intersections of the~$i$-th vertical and the~$i$-th
    horizontal line for each~$i$.
  \item Then $\LK(D)$ is the Dyck path of semilength~$n$ whose
    valleys are the marked points.
  \end{enumerate}
\end{definition}

In \Cref{fig:Lalanne-Kreweras}, the Lalanne-Kreweras involution of the mirrored black path yields the blue path.
The vertical and horizontal lines drawn in the second step are coloured green.
The black crosses should be ignored for now.

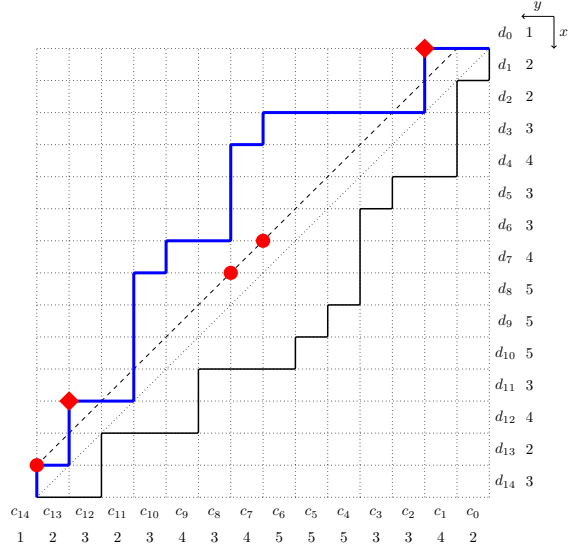
\begin{figure}
  \centering
  \resizebox{.6\textwidth}{!}{
  \begin{tikzpicture}[scale=0.75]
    \draw[dotted] (0, 0) grid (14, 14);
    \draw[dashed] (0, 1) -- (13, 14);
    \draw[dotted] (0, 0) -- (14, 14);
    \draw[->] (16, 15) -- node[above] {$y$} (15,15);
    \draw[->] (16, 15) -- node[right] {$x$} (16,14);
\drawPath%
{1, 1, 0, 0, 1, 1, 1, 0, 0, 1, 1, 1, 0, 1, 0, 1, 0, 0, 0, 1, 0, 1, 1, 0, 0, 0, 1, 0} 
{1}{black}{(0,0)}{1}
\drawPath%
{1, 0, 1, 1, 0, 0, 1, 1, 1, 1, 0, 1, 0, 0, 1, 1, 1, 0, 1, 0, 0, 0, 0, 0, 1, 1, 0, 0} 
{0}{blue}{(0,0)}{2}
\drawRegular%
{1/0,7/6,8/7} 
{3/1,14/12} 
{3, 2, 4, 3, 5, 5, 5, 4, 3, 3, 4, 3, 2, 2, 1} 
{2, 4, 3, 3, 5, 5, 5, 4, 3, 4, 3, 2, 3, 2, 1} 
  \end{tikzpicture}
  }
  \caption{A complete example for \Cref{thm:LK}.  The mirrored Dyck path $D$,
    drawn below the main diagonal, is black and thin, its Lalanne-Kreweras
    involution is blue. The area sequence for~$D$ is at the bottom,
    the coarea sequence for~$D$ at the right hand side.  The
    $1$-regular modules of the corresponding Nakayama algebra are
    $S_6$, $S_7$ and $S_{13}$.  Corresponding $1$-cuts are marked
    with a red circle. The $2$-regular modules are $S_0$ and
    $S_{11}$. Corresponding $2$-hills are marked with a red diamond.}
  \label{fig:regular-on-Dyck-paths}
\end{figure}
\begin{theorem}
\label{thm:LK}
  Let~$A$ be an~$n$-LNakayama algebra, let~$D$ be
  the Dyck path corresponding to~$A$ and let $E = \LK(D)$ be the
  image of~$D$ under the Lalanne-Kreweras involution.  Then the $A$-module
  \begin{enumerate}
  \item $S_i$ is $1$-regular if and only if $E$ has a $1$-cut at
    position~$i$.
  \item $S_i$ is $2$-regular if and only if $E$ has a $2$-hill at
    position~$i$.
  \end{enumerate}
\end{theorem}
\begin{proof}
  Suppose first that $S_i$ is $1$-regular.  By \Cref{regchara}(1),
  $c_i - c_{i+1} = d_{i+1} - d_i = 1$.  Because of
  $c_i -c_{i+1} = 1$, there is a double horizontal step in the path
  below the diagonal, whose midpoint has $y$-coordinate $i+1$.
  Because of $d_{i+1} - d_i = 1$ there is a double vertical step
  whose midpoint has $x$-coordinate~$i$.  The corresponding vertical
  and horizontal lines (coloured green in
  \Cref{fig:Lalanne-Kreweras}) intersect at the diagonal $y = x+1$,
  dashed in \Cref{fig:regular-on-Dyck-paths}.

  By the definition of the Lalanne-Kreweras involution, the Dyck path
  $\LK(D)$ has a valley at the end of every green line.  Therefore,
  for each vertical line there is a peak of $\LK(D)$ with the same
  $y$-coordinate as the line, and for each horizontal line there is a
  peak at the same $x$-coordinate as the line.  Specifically, for two
  green lines intersecting at $(i,i+1)$, there is a peak
  corresponding to the vertical line with $y$-coordinate~$i+1$, and a
  peak corresponding to the horizontal line with $x$-coordinate ~$i$,
  that is, a $1$-cut.

  Conversely, if there are two such peaks, the two corresponding
  vertical and horizontal lines intersect at the diagonal $y=x+1$,
  implying that $S_i$ is $1$-regular.

  Let us now show that $2$-regular modules correspond to hills of
  size~$2$.  We begin by noting that a hill of size~$2$ at position $i$ in $\LK(D)$ forces
  the conditions on~$D$ in \Cref{regchara}(2).  Suppose for simplicity that the $2$-hill is neither at the beginning nor at the end of $\LK(D)$, the argument is easy to adapt for these two degenerate cases.  Since $\LK(D)$ has a return at position $i$ and no return at position $i+1$, the number of double falls equals the number of double rises after the first $2(i+1)$ steps of $D$, which implies that $D$ has a return at $i+1$.  Thus, $c_{i} = d_{i+2} = 2$.  Because $\LK(D)$ has a return at position $i$, the path below the diagonal has a double vertical step whose midpoint has $x$-coordinate $i$, which translates into the condition $d_{i+1}-d_i=1$ on the coarea sequence of $D$.  Similarly, because of the return of $\LK(D)$ at position $i+2$, the path below the diagonal has a double horizonatal step whose midpoint has $y$-coordinate $i+2$, which translates into the condition $c_{i+1}-c_{i+2}=1$ on the area sequence of $D$.

  Conversely, the conditions $c_i = d_{i+2} = 2$ imply that the mirrored Dyck
  path~$D$ below the diagonal has a return to the diagonal with $x$-
  and $y$-coordinate $i+1$.

  Let us ignore the degenerate cases where~$D$ begins
  or ends with a $1$-hill.  Then, the horizontal line emanating from the midpoint of the
  double vertical step forced by $d_{i+1} = d_i + 1$ must be matched
  with the vertical line emanating from the midpoint of the last
  double horizontal step before - to the right and above - the return
  to the diagonal.  Thus, the intersection of these two lines is on
  the diagonal of~$D$.

  Similarly, the vertical line emanating from the midpoint of the
  double horizontal step forced by $c_{i+1} = c_{i+2} + 1$ must be
  matched with the horizontal line emanating from the first double
  vertical step after the return to the diagonal, and their
  intersection is on the diagonal.  Finally, we observe that the
  distance between these two intersections is $2$.
\end{proof}

In the following we describe a bijection on Dyck paths that yields an
even simpler description of the $1$- and $2$-regular simple modules.
The main ingredient is the Billey-Jockusch-Stanley bijection, which is closely
related to the Lalanne-Kreweras involution:%
\begin{definition}[\cite{BJS}]\label{dfn:BJS}
  A \Dfn{$321$-avoiding permutation} is a permutation $\pi$ such that there is no triple $i < j < k$ with $\pi(k) < \pi(j) < \pi(i)$.
  The \Dfn{Billey-Jockusch-Stanley bijection}\footnote{\FindStat{Mp00129}} $\BJS$ sends a Dyck path~$D$ of
  semilength~$n$ to a $321$-avoiding permutation $\pi$ of the numbers
  $\{1,\dots,n\}$ as follows:
  \begin{enumerate}
  \item Mirror the Dyck path~$D$ to obtain a path below the main diagonal, from the top
    right to the bottom left.
  \item Put crosses into the cells corresponding to the valleys of~$D$.
  \item\label{it:LK-peaks} Then, working from right to left, for each column not
    yet containing a cross we put a cross into the top most cell whose
    row does not yet contain a cross.
  \end{enumerate}
  Replacing all crosses with the integer $1$ and filling all other cells with the integer $0$ yields the permutation matrix of the reverse complement of $\pi$.
\end{definition}

Note that one can equivalently fill in the crosses in step~\eqref{it:LK-peaks} from left to right, putting crosses into the lowest available cell.

In \Cref{fig:Lalanne-Kreweras}, the black crosses indicate the
permutation matrix of $\BJS(D)$.  As visible there, we have the
following relation between the Lalanne-Kreweras involution and the
Billey-Jockusch-Stanley bijection:
\begin{proposition}\label{prop:LK-BJS}
  The peaks of $\LK(D)$ are at the positions of the crosses of the
  permutation matrix of $\BJS(D)$ above the main diagonal.
\end{proposition}
\begin{proof}
  Let $(i_0,j_0),\dots,(i_k,j_k)$ with $0=i_0 < \dots < i_k < n$ and $0 < j_0 < \dots < j_k = n$ be the coordinates of the peaks of $\LK(D)$.  Then, by step~(2) of \Cref{dfn:LK} of the Lalanne-Kreweras involution, $D$ has a valley corresponding to a cell in a column just to the right of $y=j$, $0 < j < n$, if and only if $j\not\in\{j_0,\dots,j_k=n\}$.  For the same reason, $D$ has a valley corresponding to a cell in a row just below $x=i$, $0 < i < n$, if and only if $i\not\in\{0=i_0,\dots,i_k\}$.

  Thus, by step~(3) of \Cref{dfn:BJS} of the Billey-Jockusch-Stanley bijection, there are crosses in the cells $(i_0,j_0),\dots,(i_k,j_k)$, because for $0\leq a\leq k$ the row just below $x=i_a$ is the top most row not containing a cross, once crosses have been placed in the cells $(i_0,j_0),\dots,(i_{a-1},j_{a-1})$.
\end{proof}

Our final reformulations of $1$- and $2$-regularity in terms of Dyck paths uses a slightly more involved bijection, but has the advantage of describing $1$- and $2$-regular statistics in a completely local way.
\begin{figure}
  \centering
  \resizebox{.6\textwidth}{!}{
  \begin{tikzpicture}[scale=0.75]
    \draw[dotted] (0, 0) grid (14, 14);
    \draw[dashed] (0, 1) -- (13, 14);
    \draw[dotted] (0, 0) -- (14, 14);
    \draw[->] (16, 15) -- node[above] {$y$} (15,15);
    \draw[->] (16, 15) -- node[right] {$x$} (16,14);
\drawPath%
{1, 1, 0, 0, 1, 1, 1, 0, 0, 1, 1, 1, 0, 1, 0, 1, 0, 0, 0, 1, 0, 1, 1, 0, 0, 0, 1, 0} 
{1}{black}{(0,0)}{1}
\drawPath%
{1, 0, 1, 1, 0, 0, 1, 1, 1, 1, 0, 1, 0, 0, 1, 1, 1, 0, 1, 0, 0, 0, 0, 0, 1, 1, 0, 0} 
{0}{blue}{(0,0)}{2}
\drawCycleDiagram%
{1/1,3/2,7/4,8/5,11/7,12/8,14/13,2/3,4/6,5/9,6/10,9/11,10/12,13/14} 
{2,4,5,13} 
{6,9,10} 
{0/1/1/2,2/3/3/5,2/4/7/10,4/6/8/10,4/7/11/13,10/12/12/13} 
\drawRegular%
{1/0,7/6,8/7} 
{3/1,14/12} 
{3, 2, 4, 3, 5, 5, 5, 4, 3, 3, 4, 3, 2, 2, 1} 
{2, 4, 3, 3, 5, 5, 5, 4, 3, 4, 3, 2, 3, 2, 1} 
  \end{tikzpicture}
  }
  \caption{The cycle diagram of the permutation associated with the
    Dyck path, together with the points indicating the $1$-~and
    $2$-regular modules.  The configurations along the diagonal
    specifying the composition are indicated with black \lq
    L\rq-shapes and circles.}
  \label{fig:EK}
\end{figure}
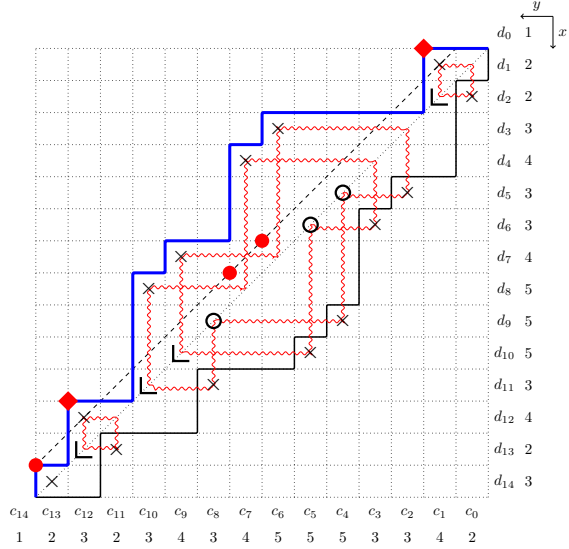
\begin{figure}
  \centering
  \resizebox{.6\textwidth}{!}{
  \begin{tikzpicture}[scale=0.75]
    \draw[dotted] (0, 0) grid (14, 14);
    \draw[dashed] (0, 1) -- (13, 14);
    \draw[dotted] (0, 0) -- (14, 14);
    \draw[->] (16, 15) -- node[above] {$y$} (15,15);
    \draw[->] (16, 15) -- node[right] {$x$} (16,14);
\drawPath%
{1, 1, 0, 0, 1, 1, 1, 0, 0, 1, 1, 1, 0, 1, 0, 1, 0, 0, 0, 1, 0, 1, 1, 0, 0, 0, 1, 0} 
{1}{black}{(0,0)}{1}
\drawPath%
{1, 0, 1, 1, 0, 0, 1, 1, 1, 1, 0, 1, 0, 0, 1, 1, 1, 0, 1, 0, 0, 0, 0, 0, 1, 1, 0, 0} 
{0}{blue}{(0,0)}{2}
\begin{scope}[shift={(-0.25,0.25)}]
  \drawPath%
  {1,0,1,1,0,0,1,1,1,1,0,1,0,1,1,1,0,0,1,0,0,0,0,0,1,1,0,0}
  {0}{green}{(0,0)}{2}
\end{scope}
\drawRegular%
{1/0,7/6,8/7} 
{3/1,14/12} 
{3, 2, 4, 3, 5, 5, 5, 4, 3, 3, 4, 3, 2, 2, 1} 
{2, 4, 3, 3, 5, 5, 5, 4, 3, 4, 3, 2, 3, 2, 1} 
  \end{tikzpicture}
  }
  \caption{A complete example for \Cref{thm:1-rises-2-hills}.  The
    Dyck path~$D$ is black and thin, its Lalanne-Kreweras involution is blue.
    The area sequence for~$D$ is at the bottom, the coarea sequence
    at the right hand side.  The $1$-regular modules of the
    corresponding Nakayama algebra are $S_6$, $S_7$ and $S_{13}$.
    Corresponding $1$-cuts of $\LK(D)$ are marked with a red
    circle. The $2$-regular modules are $S_0$ and $S_{11}$.
    Corresponding $2$-hills of $\LK(D)$ are marked with a red
    diamond.  The final path, $\phi(D)$, in green, has $1$-rises at $x$-coordinates
    $6$, $7$ and $13$, and $2$-hills at $0$ and $11$.  It is slightly
    shifted to improve visibility.}
  \label{fig:regular-on-Dyck-paths2}
\end{figure}
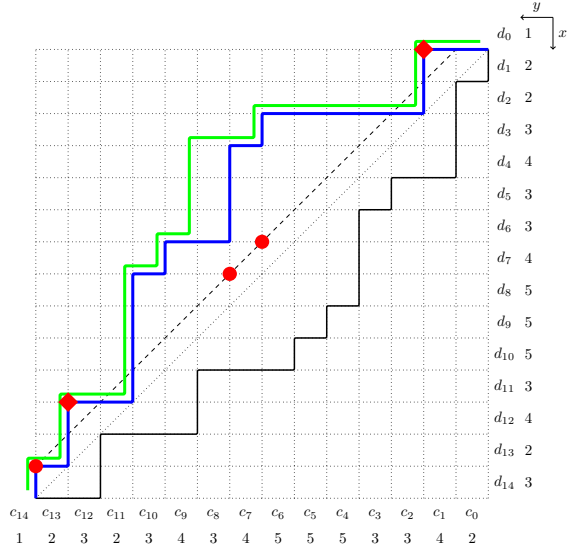
\begin{theorem}
\label{thm:1-rises-2-hills}
  Let~$A$ be an~$n$-LNakayama algebra and let~$D$ be corresponding Dyck path.
  Then there is an explicit bijection~$\phi$, such that the $A$-module
  \begin{enumerate}
    \item\label{it:1reg1rise} $S_i$ is $1$-regular if and only if
      $\phi(D)$ has a $1$-rise with $x$-coordinate~$i$.

    \item $S_i$ is $2$-regular if and only if $\phi(D)$ has a $2$-hill
    at position~$i$.
  \end{enumerate}
\end{theorem}
\begin{proof}
  Taking into account \Cref{thm:LK}, it suffices to provide a
  bijection $\psi$ on Dyck paths that preserves hills of size $2$ and
  maps $1$-cuts to $1$-rises.

  Let $E$ be a Dyck path of semilength~$n$.  We first construct
  Elizalde's \lq cycle diagram\rq\ of the $321$-avoiding permutation
  associated with $E$ by the Billey-Jockusch-Stanley bijection~\cite{Elizalde}: for each cross, draw a horizontal and a vertical line connecting the cross with the main diagonal.  For
  the Dyck path $E=\LK(D)$ in \Cref{fig:Lalanne-Kreweras}, this is
  carried out in \Cref{fig:EK}.

  We then record the sequence of configurations of lines emanating
  from the main diagonal of the cycle diagram, and construct a (weak)
  composition $\alpha$ as follows.  Points on the main diagonal with
  both lines being in the upper left of the diagram (as drawn in
  \Cref{fig:EK}) correspond to $1$-cuts, and are ignored.  If the
  horizontal line is in the upper left, and the vertical line in the
  lower right of the diagram, the point is also ignored.

  Of the remaining points, those who have their horizontal line in
  the lower right and their vertical line in the upper left of the
  diagram serve as delimiters of a composition $\alpha$, which we now construct.  In the figure
  these are indicated by black \lq L\rq-shapes.  Thus, by \Cref{prop:LK-BJS}, the number of parts $\ell(\alpha)$ of the composition is the number of peaks of $E$ minus the number of $1$-cuts of $E$.  The~$i$-th part of
  the composition, $\alpha_i$, is the number of points between
  the~$i$-th and the $(i+1)$-st delimiter with both lines in the
  lower right of the diagram.  In the figure these points are
  indicated by black circles. Thus, the composition corresponding to
  the configuration in \Cref{fig:EK} is $\alpha=(0,3,0,0)$.

  Note that the number of points with the horizontal line in the
  upper left and the vertical line in the lower right equals the
  number of points with the horizontal line in the lower right and
  the vertical line in the upper right.  Therefore, we have that
  \begin{equation}
    \label{eq:composition}
    2\ell(\alpha) + |\alpha| + c = n,
  \end{equation}
  where $|\alpha|$ is the sum of the parts of the composition $\alpha$ and $c$ is the number of $1$-cuts of $E$.

  Finally, $\psi(E)$ is the unique Dyck path that has peaks at the
  same $x$-coordinates as $E$, $1$-rises at the $x$-coordinates of
  the $1$-cuts of $E$, and the number of horizontal steps on the remaining $x$-coordinates given
  by adding $2$ to each part of $\alpha$.  This is well defined
  because of \Cref{eq:composition}.
\end{proof}

\begin{theorem}\label{thm:rectangle-return}
  A Dyck path $D$ has a rectangle at $(i+1, j)$ if and only if $\LK(D)$ has a return at position $j+1=i+c_i$, which is not the final step of a $1$-hill.
\end{theorem}
\begin{proof}
  Suppose that $D$ has a valley at $(i+1, j)$, such that its next valley has $x$-coordinate strictly larger than $j+1$.
  Thus, $\BJS(D)$ has no crosses in the region below and to the right of $(j+1,j+1)$.
  Because $\BJS(D)$ is a permutation, by the pigeonhole principle, it has no crosses in the region above and to the left of $(j+1,j+1)$ either.
  Consequently, $\LK(D)$ has no peaks in this region, and therefore $\LK(D)$ must have a return with $x$-coordinate $j+1$.
  This cannot be the second step of a $1$-hill, because there is a cross in the cell with $x$-coordinate $j$, corresponding to the valley $(i+1, j)$ of $D$, and a $1$-hill would correspond to a cross in the cell with coordinates $(j,j)$.
\end{proof}

We conclude with some corollaries enumerating LNakayama algebras with certain homological restrictions.
\begin{corollary}
\label{cor:pdim1}
  The number of $(n+1)$-LNakayama algebras with exactly~$\ell$ simple modules of projective dimension~$1$ and the number of $(n+1)$-LNakayama algebras with exactly~$\ell$ simple modules of projective dimension at least~$2$ equal the Narayana numbers\footnote{\OEIS{A001263}}, counting Dyck paths of semilength~$n$ with exactly~$\ell$ peaks.
  Explicitly, this number is
  \[
    \frac{1}{n}\binom{n}{\ell-1}\binom{n}{\ell}.
  \]
\end{corollary}
\begin{proof}
  This is a direct consequence of \Cref{smallprojdim}\eqref{it:pd1} and the fact that the number of peaks plus the number of double falls equals the semilength of the Dyck path.
\end{proof}

The proofs of several of the further corollaries involve Lagrange inversion.
\begin{theorem}[\protect{eg.,~\cite[Theorem~5.4.2]{StanleyEC2}}]
  Let $H$ be any formal power series and let $F$ be a formal power series with compositional inverse $F^{(-1)}$.  Then
  \[
  [x^n] H(F(x)) = \frac{1}{n} [x^{n-1}] H'(x) \left(\frac{x}{F^{(-1)}(x)}\right)^n,
  \]
  where $[x^n] H(x)$ is the coefficient of $x^n$ in $H(x)$.
\end{theorem}

\begin{corollary}
  \label{cor:pdim2}
  The number of $(n+1)$-LNakayama algebras with exactly~$\ell$ simple modules of projective dimension~$2$ is the number of Dyck paths of semilength~$n$ with exactly~$\ell$ returns which are not $1$-hills.  Explicitly, this number is\footnote{\OEIS{A097877}}
  \[
    \sum_{k=0}^{n-2\ell}\frac{\ell}{k+\ell}\binom{2(k+\ell)}{k}\binom{n-k-\ell}{\ell}.
  \]
\end{corollary}
\begin{proof}
  The claim in the first sentence follows from \Cref{smallprojdim}\eqref{it:pd2} and \Cref{thm:rectangle-return}.  To enumerate these, let $D(x) = 1 + x D(x)^2$ be the generating function for all Dyck paths.  Then, $x\big(D(x)-1\big)$ is the generating function for Dyck paths without $1$-hills.  Since $1/(1-x)$ is the generating function for (possibly empty) paths consisting only of $1$-hills,
  \[
    \frac{x^\ell \big(D(x)-1\big)^\ell}{(1-x)^{\ell+1}}
  \]
  is the generating function for all Dyck paths with exactly~$\ell$
  returns which are not $1$-hills.

  Using Lagrange inversion we find that the coefficient of $x^k$ in $\big(D(x)-1\big)^\ell$  equals
  $
  \frac{\ell}{k}\binom{2k}{k-\ell},
  $
  and using the binomial theorem we find that the coefficient of $x^{n-\ell-k}$ in $(1-x)^{-\ell-1}$ equals
  $
  \binom{n-k}{n-k-\ell}.
  $
\end{proof}

\begin{corollary}
  \label{cor:pdim1pdim2GF}
  Let $a_{n,k,\ell}$ be the number of~$(n+1)$-LNakayama algebras with $k$
  simple modules of projective dimension~$1$ and $\ell$~simple
  modules of projective dimension~$2$ and let
  \begin{align*}
    N(x,q,t) &= \sum_{n,k,\ell} a_{n,k,\ell} x^n q^k t^\ell\\%
    &= 1 + q x + (q^2+qt)x^2 + (q^3+3q^2t+qt)x^3 + \cdots
  \end{align*}
  Then
  \begin{multline*}
  \Big(x^2(q-t)q(t-1) - x(2qt - 2q + t) + t - 1\Big) N(x,q,t)^2\\%
  +\Big((qt - 2q + t)x - t + 2\Big) N(x,q,t)%
  - 1 = 0.
  \end{multline*}
\end{corollary}
\begin{proof}
  According to \Cref{smallprojdim} we want to count Dyck paths with $k$ double falls and $\ell$ rectangles.  Using the definition of the Lalanne-Kreweras involution and \Cref{thm:rectangle-return}, we can equivalently count Dyck paths with $k$ peaks and $\ell$ returns which are not $1$-hills.

  Let $d_{n,k,\ell}$ be the number of Dyck paths of semilength~$n$ with $k$ peaks and $\ell$ returns which are not $1$-hills, and let $D(x, q, t)=\sum_{n\geq 0} d_{n,k,\ell} x^n q^k t^\ell$ be the corresponding generating function.  By the foregoing, $D(x, q, t) = N(x, q, t)$.

  In the following we will frequently use the so called \lq first passage decomposition\rq\ of Dyck paths: we decompose a non-empty Dyck path into an initial Dyck path, which has a single return (which is its final step), and a remaining Dyck path.

  Since a Dyck path is either empty, or begins with a $1$-hill (which is a peak), or begins with a horizontal step, followed by a non-empty Dyck path, followed by a vertical step (which is a return, and not a $1$-hill), $D(x, q, t)$ satisfies the equation
  \[
  D(x, q, t) = 1 + xq D(x, q, t) + xt\big(D(x, q, 1) - 1\big) D(x, q, t).
  \]
  Substituting $t=1$ we obtain a quadratic equation for $D(x, q, 1)$, with a unique solution which is a formal power series (and not a Laurent series).  It is then straightforward to check that $D(x, q, t) = 1/\big(1-xq-xt(D(x, q, 1) - 1)\big)$ satisfies the claimed equation.
\end{proof}

\begin{corollary}
\label{cor:no1reg}
  The number of $(n+1)$-LNakayama algebras without $1$-regular simple modules equals the Riordan number\footnote{\OEIS{A005043}}, counting Dyck paths of semilength~$n$ without $1$-rises.
  Explicitly, this number is
  \[
    \frac{1}{n+1}\sum_{k=0}^n \binom{n+1}{k}\binom{n-k-1}{k-1}.
  \]
\end{corollary}

\begin{corollary}
\label{cor:no2reg}
  For $n \geq 1$, the number of $(n+1)$-LNakayama algebras~$A$ without $2$-regular simple modules (that is, such that the category of finitely generated projective modules has a unique exact structure) equals the number of Dyck paths without $2$-hills\footnote{\OEIS{A114487}}.
  Explicitly, this number is
  \[
    \sum_{k=0}^{\lfloor\frac{n}{2}\rfloor} %
    (-1)^k\frac{k+1}{n-k+1}\binom{2n-3k}{n-k}.
  \]
\end{corollary}
\begin{proof}
  The formula for the number of such Dyck paths was provided by Ira Gessel~\cite{MO}.
\end{proof}

\begin{corollary}
  \label{cor:k1regl2reg}
  Let $a_{n,k,\ell}$ be the number of~$(n+1)$-LNakayama algebras with $k$
  simple $1$-regular and $\ell$ simple $2$-regular modules and let
  \begin{align*}
    N(x,q,t) &= \sum_{n,k,\ell} a_{n,k,\ell} x^n q^k t^\ell\\%
    &= 1 + q x + (q^2+t)x^2 + (q^3+2qt+q+1)x^3 + \cdots
  \end{align*}
  Then
  \begin{multline*}
  \Big(x^3 (t-1)^2 + x^2 (t-1)(q-1) - x (t-1 + q-1) +1\Big) N(x,q,t)^2\\%
  +\Big(2x^2(t-1) + x(q-1) -1\Big) N(x,q,t)%
  + x = 0.
  \end{multline*}
\end{corollary}
\begin{proof}
  Following to \Cref{thm:1-rises-2-hills} it suffices to determine the number $d_{n,k,\ell}$ of Dyck paths of semilength~$n$ with $k$ $1$-rises and $\ell$ $2$-hills.  Let $D(x, q, t)=\sum_{n\geq 0} d_{n,k,\ell} x^n q^k t^\ell$ be the corresponding generating function.  By the foregoing, $D(x, q, t) = N(x, q, t)$.

  $D(x, q, 0)$ is the generating function for Dyck paths without $2$-hills.  Since a Dyck path either contains no $2$-hills, or begins with a Dyck path without $2$-hills, followed by a $2$-hill, we have the equation
  \[
  D(x, q, t) = D(x, q, 0) + D(x, q, 0) x^2t D(x, q, t).
  \]
  To obtain an equation for $D(x, q, 0)$, we observe that a Dyck path without $2$-hills is either empty, begins with a $1$-hill, or begins with a double rise.  The generating function for prime Dyck paths beginning with a double rise, equals
  \[
  D(x, q) = x^2 \big(D(x, q, 1)-1\big) + x \big(D(x, q, 0) (1-xq) - 1\big),
  \]
  where we distinguish whether there is a peak immediately after the double rise or not.
  Thus,
  \[
  D(x, q, 0) = 1 + xq D(x, q, 0) + DD(x, q) D(x, q, 0).
  \]
  From these equations we can compute $D(x, q, t)$, and check that it satisfies the claimed equation.
\end{proof}

Using \Cref{thm:enomoto}, \Cref{thm:1-rises-2-hills} also gives a sharp upper bound for the number of exact structures on the category of finitely generated projective modules for $n$-LNakayama algebras.

\begin{corollary}
\label{bound2reglnak}
  An $n$-LNakayama algebra has at most $\lfloor\frac{n-1}{2}\rfloor$ $2$-regular simple modules and thus at most $2^{\lfloor\frac{n-1}{2}\rfloor}$ exact structures on the category of finitely generated projective modules.
  This bound is sharp.
\end{corollary}
\begin{proof}
  A Dyck path of semilength $n-1$ has at most $\lfloor\frac{n-1}{2}\rfloor$ $2$-hills.  Following \Cref{thm:LK}, the bound is thus obtained for the $n$-LNakayama algebras with Kupisch series $[2,3,\dots,2,3,2,2,1]$ if $n$ is odd, and, for example, $[2,3,\dots,2,3,2,1]$ if $n$ is even.
\end{proof}

\subsubsection{Describing 1-regular simple modules using the zeta map}
\label{sec:zetamap}

We finish this section with an alternative approach to \Cref{thm:1-rises-2-hills}\eqref{it:1reg1rise} using the \emph{zeta map}.
We refer to~\cite[page~50]{Hag2008} for the history of this map and its original context.
Let~$D$ be a Dyck path of semilength~$n$ with coarea sequence $(d_0,\ldots,d_{n})$.  We obtain a Dyck path~$\zeta(D)$ as follows:
\begin{itemize}
  \item First, let~$a_k$ be the number of indices~$i$ with $d_i = k$ and build an intermediate Dyck path (the \Dfn{bounce path}) consisting of~$a_2$ horizontal steps, followed by~$a_2$ vertical steps, followed by~$a_3$ horizontal and vertical steps, and so on.

  \item Then, we fill the rectangular regions between two consecutive peaks of the bounce path.
  Observe that the rectangle between the $k$-th and the $(k+1)$-st peak must be filled by $a_{k+1}$ vertical steps and $a_{k+2}$ horizontal steps.
  We do so by scanning the coarea sequence $(d_0,\dots,d_n)$ and drawing a vertical or a horizontal step whenever we encounter a $k+1$ or a $k+2$, respectively.
\end{itemize}

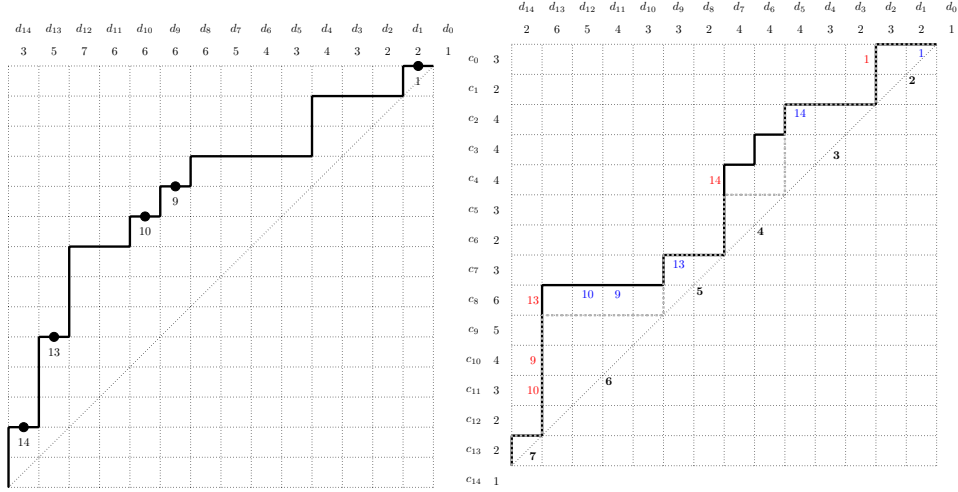
\begin{figure}[t]
  \resizebox{\textwidth}{!}{
    \begin{tikzpicture}[scale=0.9]
      \draw[dotted] (0,  0) grid (14, 14);
      \draw[dotted] (0,  0) --   (14, 14);
      \drawPath%
      {1,1,0,1,1,1,0,1,1,1,0,0,1,0,1,0,1,0,0,0,0,1,1,0,0,0,1,0} 
      {0}{black}{(0,0)}{2}
      \drawLabel%
      {}{}
      {\coa}   {3,5,7,6,6,6,6,5,4,3,4,3,2,2,1} 
      \draw[fill=black] (0.5, 2) circle (.15);
      \node at ( 0.5, 1.5) {14};
      \draw[fill=black] (1.5, 5) circle (.15);
      \node at ( 1.5, 4.5) {13};
      \draw[fill=black] ( 4.5, 9) circle (.15);
      \node at ( 4.5, 8.5) {10};
      \draw[fill=black] ( 5.5,10) circle (.15);
      \node at ( 5.5, 9.5) {9};
      \draw[fill=black] (13.5,14) circle (.15);
      \node at (13.5,13.5) {1};
    \end{tikzpicture}
    \hfill
    \begin{tikzpicture}[scale=0.9]
      \draw[dotted] (0,  0) grid (14, 14);
      \draw[dotted] (0,  0) --   (14, 14);
      \drawPath%
      {1,0,1,1,1,1,1,0,0,0,0,1,0,0,1,1,1,0,1,0,1,0,0,0,1,1,0,0} 
      {0}{black}{(0,0)}{2}
      \drawPath%
      {1,0,1,1,1,1,0,0,0,0,1,1,0,0,1,1,0,0,1,1,1,0,0,0,1,1,0,0} 
      {0}{lightgrey, dotted, line width=2}{(0,0)}{2}
      \drawLabel%
      {c}      {3,2,4,4,4,3,2,3,6,5,4,3,2,2,1} 
      {\coa}   {2,6,5,4,3,3,2,4,4,4,3,2,3,2,1} 
      \node at (13.2,12.8) {\bf 2};
      \node at (10.7,10.3) {\bf 3};
      \node at ( 8.2, 7.8) {\bf 4};
      \node at ( 6.2, 5.8) {\bf 5};
      \node at ( 3.2, 2.8) {\bf 6};
      \node at ( 0.7, 0.3) {\bf 7};
      \node[blue] at ( 2.5, 5.7) {10};
      \node[blue] at ( 3.5, 5.7) {9};
      \node[blue] at ( 5.5, 6.7) {13};
      \node[blue] at ( 9.5,11.7) {14};
      \node[blue] at (13.5,13.7) {1};
      \node[red] at ( 0.7, 2.5) {10};
      \node[red] at ( 0.7, 3.5) {9};
      \node[red] at ( 0.7, 5.5) {13};
      \node[red] at ( 6.7, 9.5) {14};
      \node[red] at (11.7,13.5) {1};
    \end{tikzpicture}
  }
  \caption{A Dyck path of semilength~$14$ and its image $\zeta(D)$ under the zeta map.}
  \label{fig:zetamap}
\end{figure}
In the example in \Cref{fig:zetamap}, a Dyck path~$D$ (on the left) with coarea sequence
\[
  [1,2,2,3,4,3,4,5,6,6,6,6,7,5,3]
\]
and its image~$\zeta(D)$ (on the right) under the zeta map is shown.
In dotted grey, the intermediate bounce path is shown.

\medskip

For a given Dyck path with coarea sequence $(d_0,\dots,d_n)$, the definition of the zeta map yields a labelling of the vertical steps and of the horizontal steps of~$\zeta(D)$ with the indices $\{ 1,\ldots,n\}$ by associating to $1 \leq j \leq n$ the vertical and the horizontal step drawn using the entry~$d_j$.
In the example in \Cref{fig:zetamap}, the vertical steps are labelled from top to bottom by the permutation $[1,2,3,5,14,4,6,7,13,8,9,10,11,12]$ as are the horizontal steps from right to left.
In symbols and in terms of the inverse permutation, the vertical step of~$\zeta(D)$ labelled~$j$ for $1 \leq j \leq n$ has initial $x$-coordinate
\[
  k(j) = \#\big\{0 \leq i \leq n\ :\ d_i < d_j \big\} + \#\big\{ 0 \leq i < j\ :\ d_i = d_j \big\} - 1
\]
for the coarea sequence $(d_0,\dots,d_n)$ of~$D$, and the horizontal step labelled by~$j$ has final $y$-coordinate $k(j) + 1$.
In the example, the $k(j)$ for $1 \leq k \leq 14$ is given by $[0,1,2,5,3,6,7,9,10,11,12,13,8,4]$.

\medskip

We then have the following alternative to \Cref{thm:1-rises-2-hills}\eqref{it:1reg1rise}.

\begin{theorem}
\label{thm:1-regzeta}
  Let~$D$ be a Dyck path of semilength~$n$ and let~$A$ be the Nakayama algebra corresponding to~$\zeta(D)$.
  Then~$D$ has
  \begin{itemize}
    \item a peak with $y$-coordinate~$j$ if and only if the simple~$A$-module~$S_{k(j)}$ has projective dimension~$1$, and
    \item a $1$-rise with $y$-coordinate~$j$ if and only if~$S_{k(j)}$ is $1$-regular.
  \end{itemize}
\end{theorem}
We remark that the number of $1$-regular modules in LNakayama algebras and the number of $1$-rises in Dyck paths seem to have a symmetric joint distribution.
It may be interesting to find an appropriate bijection.

\medskip

The crucial observation for the proof of \Cref{thm:1-regzeta} is the following lemma.
\begin{lemma}
\label{prop:pd1zeta}
  Let~$D$ be a Dyck path of semilength~$n$ and let $(c_0,\dots,c_n)$ and $(d_0,\dots,d_n)$ be the area and, respectively, the coarea sequence of $\zeta(D)$.
  Then,
  \begin{itemize}
   \item for any $1 \leq j \leq n$, the path~$D$ has a peak with $y$-coordinate~$j$ if and only if $c_{k(j)} - c_{k(j)+1} = 1$, and
   \item for any $2 \leq j \leq n$, the path~$D$ has a valley with $y$-coordinate~$j-1$ if and only if $d_{k(j)+1} - d_{k(j)} = 1$.
  \end{itemize}
\end{lemma}

\begin{proof}
  Observe first that for any $1 \leq j \leq n$, the vertical step of~$\zeta(D)$ labelled with~$j$ corresponds to the entry~$c_{k(j)}$ of its area sequence.

  Let now $(d_0',\dots,d_n')$ be the coarea sequence of~$D$ and let~$1 \leq j \leq n$.
  Then~$D$ has a peak with $y$-coordinate~$j$ if and only if $d_j' \geq d_{j+1}'$.
  This is the case if and only if the vertical step of $\zeta(D)$ labelled with~$j$ is followed by another vertical step.
  By definition, this is the case if and only if $c_{k(j)} - c_{k(j)+1} = 1$.
  This proves the claim in the first bullet point.

  The claim in the second bullet point follows from the same argument applied to horizonal steps instead of vertical steps.
\end{proof}

\begin{proof}[Proof of \Cref{thm:1-regzeta}]
  Let $2 \leq j \leq n$.
  Then~$D$ has a $1$-rise with $y$-coordinate~$j$ if and only if it has both a peak with $y$-coordinate~$j$ and also a valley with $y$-coordinate~$j-1$.
  The statement now follows from~\Cref{smallprojdim}\eqref{it:pd1} and \Cref{regchara}\eqref{it:1reg}.
  The boundary case $j = 1$ follows from the observation that $k(1) = 0$, implying $d_{k(j)} = 1$ and $d_{k(j)+1} = 2$.
\end{proof}

In the example in \Cref{fig:zetamap}, the $1$-rises in~$D$ are marked in columns $1,9,10,13,14$.
For each $1$-rise, the corresponding horizontal and the corresponding vertical step is marked with the given letter inside~$\zeta(D)$ in blue and in red, respectively.
This means that the Nakayama algebra for~$\zeta(D)$ has $1$-regular simple module
\[
  \{ S_{k(1)}, S_{k(9)}, S_{k(10)}, S_{k(13)}, S_{k(14)} \} = \{ S_0, S_{10}, S_{11}, S_8, S_4\}
\]
and simple modules
\[
  \{ S_{k(4)}, S_{k(8)}, S_{k(12)} \} = \{ S_{5}, S_{9}, S_{13} \}
\]
of projective dimension~$1$ that are not $1$-regular.

\subsection{Description in terms of Dyck path statistics for CNakayama algebras}

To extend \Cref{thm:LK} to CNakayama algebras, we introduce an
analogue of the Lalanne-Kreweras involution for certain periodic Dyck
paths.

Let us first specify a map $\LKP$ on the set $\DyckP_n$ of $n$-periodic
Dyck paths with global shift~$0$ and non-constant area sequence.
Given a path~$D$ in this set we essentially use \Cref{dfn:LK} to
construct $\LKP(D)$.  For item~(3) of this definition, we fix any
return of~$D$ to the diagonal, and stipulate that we mark the
intersection of the~$i$-th vertical line \emph{after this return}
with the~$i$-th horizontal line \emph{after this return}.

Since the number of double rises equals the number of double falls between any two returns of $D$, this definition does not depend on the return chosen.
Let us emphasize however, that $\LKP(D)$ is not necessarily in $\DyckP_n$.  For example, the image of $[3,3,2]_\circlearrowright$ equals $[5,4,3]_\circlearrowright$, which has global shift $1$.

To circumvent this defect, let $\DyckH_n$ be the set of $n$-periodic Dyck paths that have a rectangle as defined in \Cref{def:Dyck-statistics}.  We will see below that the image of $\LKP$ is exactly $\DyckH_n$.
Moreover, we will see that the CNakayama algebras corresponding to $\DyckH_n$ are precisely those which are quasi-hereditary.

Let us now describe the inverse $\LKH$ of $\LKP$ explicitly.  Again, we
essentially use \Cref{dfn:LK} to construct the image of a path~$D$ in
$\DyckH_n$.  However, since~$D$ may not have any returns to the
diagonal, we have to make item~(3) of the definition precise in a
different way.  Specifically, we fix any index~$j$ such that~$D$ has a
rectangle with $y$-coordinate~$j$.
We then stipulate that the \lq first\rq\ horizontal line
has $x$-coordinate $j+1$, and the \lq first\rq\ vertical line has
$y$-coordinate $j+1$.  In particular, the image $\LKH(D)$ of~$D$ has a return
at $j+1$.

\begin{theorem}\label{thm:LK-periodic}
  Let $\DyckH_n$ be the set of~$n$-periodic Dyck paths with a
  rectangle.  Then the map $\LKP$ is a bijection between $\DyckP_n$
  and $\DyckH_n$, with inverse $\LKH$.  Moreover, $\LKP$ is an
  involution on $\DyckP_n\cap\DyckH_n$.
\end{theorem}

Essentially, this theorem allows us to extend $\LKP$ to a map on the
union $\DyckP_n\cup\DyckH_n$.  Of course, whenever two maps agree on
the intersection of their domains, one can regard them as a single
map.  However, in the case at hand this is particularly interesting,
because the definitions of $\LKP$ and its inverse $\LKH$ are so
similar.
\begin{definition}
  The \Dfn{generalized Lalanne-Kreweras involution} $\LK$ is the map
  \begin{align*}
    \LK&: \DyckP_n\cup\DyckH_n\to\DyckP_n\cup\DyckH_n\\
    D&\mapsto
    \begin{cases}
      \LKP(D) & \text{if $D\in\DyckP_n$}\\
      \LKH(D) & \text{if $D\in\DyckH_n$.}
    \end{cases}
  \end{align*}
\end{definition}
This is well-defined, because $\LKP$ is an involution on
$\DyckP_n\cap\DyckH_n$.

\begin{proof}[Proof of~\Cref{thm:LK-periodic}]
  Let $D\in\DyckP_n$ and let $\tilde D\in\Dyck_{3n}$ be the Dyck path obtained from $D$ by restricting it to $3$ periods, ending with a return which is not the final step of a $1$-hill.
  Such a return exists because $D$ has non-constand area sequence.
  By \Cref{thm:rectangle-return}, $\tilde E = \LK(\tilde D)$ has a rectangle with $y$-coordinate $n-1$.

  By definition of the classical Lalanne-Kreweras involution and the
  definition of $\LKP$, the positions of the valleys of $\tilde E$
  coincide with the positions of the valleys of $E = \LKP(D)$ in the
  corresponding region - there are only additional peaks in
  $\tilde E$ at the beginning and the end of the period.

  In particular, $E$ also has a rectangle with $y$-coordinate $n-1$, and therefore, the mirrored path has a double horizontal step whose midpoint has $y$-coordinate $n$, and a double vertical step whose midpoint has $x$-coordinate $n$.
  Let $\LKH_n(E)$ be the path constructed from $E$ by specifying that the \lq first\rq\ horizontal line has $x$-coordinate $n$ and the \lq first\rq vertical line has $y$-coordinate $n$.
  Thus, also the following horizontal and vertical lines used to construct $\LKH_n(E)$ match up in the same way as they do to construct $\LK(\tilde E)$.
  In particular, $\LKH_n(E)$ and $\LK(\tilde E)$ coincide between $(n,n)$ and $(2n,2n)$.
  Since $\LKH_n(E)$ is determined by this region, and $\LK$ is an involution, $\LKH_n$ is indeed the inverse of $\LKP$, and $\LKP$ is an involution on $\DyckP_n\cap\DyckH_n$.

  To see that $\LKH_n$ does not depend on the rectangle chosen, note that, by \Cref{thm:rectangle-return}, each rectangle of $\tilde E$ corresponds to a return of $\tilde D$, which in turn corresponds to a return of $D$.
\end{proof}

Let us now turn to a description of quasi-hereditary Nakayama algebras.
We briefly recall the general definition, and then give an alternative description for the case of Nakayama algebras.

Let~$A$ be a quiver algebra and let $e\defeq(e_1,e_2,\dots,e_n)$ denote an ordered complete set of primitive orthogonal idempotents of~$A$, where complete means that ${\bf 1}_A=\sum\limits_{k=1}^{n}{e_k}$.
For $i \in \{1,\dots,n\}$, set $\epsilon_i\defeq e_i+e_{i+1}+\dots+e_n$, and also set $\epsilon_{n+1}\defeq 0$.
Moreover, define the \Dfn{right standard modules} $\Delta(i)\defeq e_i A/e_iA \epsilon_{i+1}A$ and dually the \Dfn{left standard modules} $\Delta(i)^{op}$ as the right standard modules of the opposite algebra of~$A$.
Define the \Dfn{right costandard modules} then as $\nabla(i)\defeq D( \Delta(i)^{op})$.
An algebra~$A$ is then called \Dfn{quasi-hereditary} in case there is an ordering $e\defeq(e_1,e_2,\dots,e_n)$ such that $\End_A(\Delta(i))$ is a division algebra for all~$i$ and $\Ext_A^2(\Delta(i),\nabla(j))=0$ for all $i$ and $j$.

Note that we used here one of the many characterizations of quasi-hereditary algebras and we refer~\cite[Theorem~A.2.6]{DK} for many more equivalent characterizations.
It is well known that any quiver algebra with an acyclic quiver is quasi-hereditary and thus every LNakayama algebra is quasi-hereditary.
Not all CNakayama algebras are quasi-hereditary, but there is an easy homological characterization as the next proposition shows.
We remark that the more general class of standardly stratified Nakayama algebras has been recently classified in~\cite{MM}.
\begin{proposition}[\protect{\cite[Proposition~3.1]{UY}}]
\label{quasi-heredchara}
 A CNakayama algebra is quasi-hereditary if and only if it has a simple module of projective dimension~$2$.
\end{proposition}
Thus, by \Cref{smallprojdim}, the CNakayama algebras corresponding to $\DyckH_n$ are precisely those which are quasi-hereditary.
The new description and~\Cref{cor:Cmin} yields their number.
\begin{corollary}
\label{cor:quasihercount}
  For any $c\geq 0$, there is an explicit bijection between
  quasi-hereditary~$n$-CNakayama algebras and~$n$-CNakayama algebras
  whose Kupisch series is non-constant and has minimal entry $c+2$.
  In particular, the number of quasi-hereditary~$n$-CNakayama
  algebras is
  \[
  \frac{1}{2n}\sum_{k\divides n}\phi\big(n/k\big)\binom{2k}{k}-1.
  \]
\end{corollary}

As an aside, we compute the size of $\DyckP_n\cap\DyckH_n$.  To do so, we recall the \Dfn{cycle construction}.
\begin{theorem}[\protect{eg.~\cite[Equation~1.4(18)]{BLL}}]
  Consider a family of mutually disjoint finite sets $(D_{n, \ell})_{n,\ell\in\NN}$ and let $D(x, q)=\sum_{n,\ell} |D_{n, \ell}| x^n q^k$ be its generating function.
  For $n,\ell\in\NN$, let $C_{n, \ell}$ be the set of cycles
  \[
  \{[d_1,\dots,d_k]_\circlearrowright \mid d_i \in D_{n_i, \ell_i}, \sum_i n_i = n, \sum_i \ell_i = \ell\}.
  \]
  Then
  \[
  |C_{n,\ell}| = \sum_{k\divides\gcd(\ell,n)}\frac{\phi(k)}{k} [x^{n/k} q^{\ell/k}] \log\frac{1}{1-D(x, q)},
  \]
  where~$\phi$ is Euler’s totient.
\end{theorem}
\begin{proposition}
  \label{prop:quasiherminentrycount}
  The number of quasi-hereditary~$n$-CNakayama algebras whose Kupisch
  series have minimal entry $2$ equals
  \[
  \frac{1}{n}\sum_{k\divides n}\phi\big(n/k\big)\sum_{m=0}^{\lfloor\frac{k}{2}\rfloor}
  \binom{2k-2m-2}{k-2}.
  \]
\end{proposition}
\begin{proof}
  Let us call an area sequence
  $[c_0,\dots,c_{n-1}]_\circlearrowright$ \Dfn{primitive}, if
  (without loss of generality) $c_{n-1}=2$ and $c_i>2$ for
  $i\neq n-1$.  Note that the concatenation of primitive area
  sequences has no rectangle if and only if none of the factors has a
  rectangle.  Thus, it is sufficient to count primitive area
  sequences without rectangle, and apply the cycle construction.

  The number of primitive area sequences of length~$n$ without
  rectangles equals the number of $321$-avoiding permutations without
  fixed points, counted by the Fine
  numbers\footnote{\OEIS{A000957}}.  This can be seen by
  interpreting $[c_0-1,\dots,c_{n-1}-1]$ as the area sequence of a
  Dyck path, and applying the Billey-Jockusch-Stanley bijection.  Fixed points
  in the resulting permutation then correspond to rectangles.
\end{proof}

\begin{theorem}
  \label{thm:LKC}
  Let~$A$ be an~$n$-CNakayama algebra, and let~$D$ be the
  corresponding $n$-periodic Dyck path.  Suppose that $D\in\DyckP_n\cup\DyckH_n$.  Then the $A$-module
  \begin{enumerate}
  \item $S_i$ is $1$-regular if and only if $\LK(D)$ has a $1$-cut at
    position~$i$,
  \item $S_i$ is $2$-regular if and only if $\LK(D)$ has a $2$-hill
    at position~$i$.
  \end{enumerate}
\end{theorem}
\begin{proof}
  The proof of \Cref{thm:LK} applies verbatim.
\end{proof}

\begin{corollary}
  Let~$A$ be an~$n$-CNakayama algebra, and let~$D$ be the corresponding $n$-periodic Dyck path.
  Suppose that $A$ has a $2$-regular simple module.  Then $D\in\DyckP_n\cap\DyckH_n$.
\end{corollary}
\begin{proof}
  Suppose that $S_i$ is $2$-regular for some $i$.
  Then, by \Cref{regchara}(2), $c_i=2$ and therefore $D\in\DyckP_n$.
  By \Cref{thm:LKC}, $\LK(D)$ has a $2$-hill, so in particular $\LK(D)\in\DyckP_n$, and, by \Cref{thm:LK-periodic}, $D\in\DyckH_n$.
\end{proof}
We note that there are $n$-CNakayama algebras with $1$-regular simple modules such that the corresponding $n$-periodic Dyck path is not even in $\DyckP_n\cup\DyckH_n$.
An example is the $2$-CNakayama algebra with Kupisch series $[4,3]$.

\begin{remark}
  There is an alternative way to extend the map~$\LK = \LKP = \LKH$ on $\DyckP_n\cap\DyckH_n$ to $\DyckP_n\cup\DyckH_n$ as follows.
  For a given $n$-periodic Dyck path~$D$ with area sequence $[a_0,\dots,a_{n-1}]_\circlearrowright$ with global shift~$c$, let $\tilde D$ be the corresponding periodic Dyck path with global shift~$0$ and area sequence $[a_0-c,\dots,a_{n-1}-c]_\circlearrowright$.
  One may now define an involution on periodic Dyck paths with global shift~$c$ for which the associated path with global shift~$0$ lies inside $\DyckP_n\cap\DyckH_n$ by mapping this associated path via the involution $\LK$ and then adding back the global shift.
  Observe that this map preserves the global shift, but does not coincide with $\LK$ outside of $\DyckP_n\cap\DyckH_n$.
  More importantly, one cannot replace $\LK$ with this definition in \Cref{thm:LKC}.

  For example, let $D\in\DyckH_3$ be the $3$-periodic Dyck path with area sequence $[5,4,3]_\circlearrowright$.
  Then $\LK(D)\in\DyckP_3$ has area sequence $[3,2,3]_\circlearrowright$, whereas the construction just outlined yields the area sequence $[3,3,4]_\circlearrowright$.
  The unique $1$-regular module of the CNakayama algebra corresponding to $D$ is $S_0$, and indeed $\LK(D)$ has a unique $1$-cut at position $0$.
  By contrast, the $3$-periodic Dyck path with area sequence $[3,3,4]_\circlearrowright$ has $1$-cuts at positions $0$ and $1$.
\end{remark}

We conclude with some corollaries enumerating CNakayama algebras with certain homological restrictions.  All of these are obtained using the bijection between quasi-hereditary CNakayama algebras and periodic Dyck paths with global shift~$0$.

\begin{corollary}
\label{cor:Cpdim1}
  The number of quasi-hereditary $n$-CNakayama algebras with exactly~$\ell<n$ simple modules of projective dimension~$1$ equals the number of $n$-periodic Dyck paths with global shift~$0$ and exactly~$\ell$ peaks.
  Explicitly, this number is
  \[
    \frac{1}{n}\sum_{k\divides\gcd(\ell,n)}\phi(k)\binom{n/k-1}{\ell/k-1}\binom{n/k}{\ell/k}.
  \]
\end{corollary}
\begin{proof}
  Let $d_{n,\ell}$ be the number of prime Dyck paths of semilength $n$ with $\ell$ peaks, and let $D(x, q) = \sum_{n,\ell} d_{n,\ell} x^n q^\ell$ be the corresponding generating function.

  Let $\tilde n=n/k$ and $\tilde\ell=\ell/k$.
  According to the cycle construction, we have to compute $[x^{\tilde n} q^{\tilde\ell}]\log\frac{1}{1-D(x, q)}$.
  Note that $D(x,q) = x\left(q + \frac{D(x,q)}{1-D(x,q)}\right)$, because it is either a $1$-hill, or a horizontal step followed by a non-empty sequence of prime Dyck paths and a vertical step.  Therefore, $D^{(-1)}(x,q) = \frac{x}{q-\frac{x}{1-x}}$.

  Using Lagrange inversion and the binomial theorem we obtain
  \begin{align*}
    \frac{1}{k}[x^{\tilde n} q^{\tilde\ell}]\log\frac{1}{1-D(x, q)} %
    &= \frac{1}{n} [x^{\tilde n-1} q^{\tilde\ell}]\frac{1}{1-x}\left(q + \frac{x}{1-x}\right)^{\tilde n}\\ %
    &= \frac{1}{n}\binom{\tilde n}{\tilde\ell} [x^{\tilde n-1}] \frac{x^{\tilde n-\tilde\ell}}{(1-x)^{\tilde n-\tilde\ell+1}}\\%
    &= \frac{1}{n}\binom{\tilde n}{\tilde\ell} [x^{\tilde\ell-1}] \frac{1}{(1-x)^{\tilde n-\tilde\ell+1}}\\%
    &= \frac{1}{n}\binom{\tilde n}{\tilde\ell} \binom{\tilde n - 1}{\tilde\ell -1}.
  \end{align*}
\end{proof}

\begin{corollary}
\label{cor:Cpdim2}
  The number of quasi-hereditary $n$-CNakayama algebras with exactly~$\ell > 0$ simple modules of projective dimension~$2$ equals the number of $n$-periodic Dyck paths with global shift~$0$ and exactly~$\ell$ returns which are not $1$-hills.
  Explicitly, this number is
  \[
    \sum_{k\divides\gcd(\ell,n)}\frac{\phi(k)}{k}\sum_{m=0}^{(n-2\ell)/k}\frac{1}{m+\ell/k}\binom{2(m+\ell/k)}{m}\binom{(n-\ell)/k-m-1}{\ell/k-1}.
  \]
\end{corollary}
\begin{proof}
  Let $D(x) = 1 + x D(x)^2$ be the generating function for all Dyck paths.
  Furthermore, let $d_{n,\ell}$ be the number of prime Dyck paths of semilength $n$ with $\ell$ returns which are not $1$-hills, and let $R(x, q) = \sum_{n,\ell} d_{n,\ell} x^n q^\ell$ be the corresponding generating function.

  Let $\tilde n=n/k$ and $\tilde\ell=\ell/k$.
  According to the cycle construction, we have to compute $[x^{\tilde n} q^{\tilde\ell}]\log\frac{1}{1-R(x, q)}$.
  Note that $R(x, q) = q x D(x) - q x + x$.

  Thus, using the expansion of the logarithm, we have
  \begin{align*}
    [x^{\tilde n} q^{\tilde\ell}]\log\frac{1}{1-R(x, q)} %
    &= [x^{\tilde n} q^{\tilde\ell}] \log\left(\frac{\frac{1}{1-x}}{1-qx\frac{D(x)-1}{1-x}}\right)\\%
    &= \frac{1}{\tilde\ell}[x^{\tilde n}]\left(x\frac{D(x)-1}{1-x}\right)^{\tilde\ell}.
  \end{align*}
  Recall from the proof of \Cref{cor:pdim2} that
  \[
  [x^{\tilde n}] \frac{x^\ell \big(D(x)-1\big)^\ell}{(1-x)^{\ell+1}} = %
  \sum_{m=0}^{n-2\tilde\ell}\frac{\tilde\ell}{m+\tilde\ell}\binom{2(m+\tilde\ell)}{m}\binom{\tilde n-m-\tilde\ell}{\tilde\ell}.
  \]
  Now, using $\binom{\tilde n-m-\tilde\ell}{\tilde\ell} - \binom{\tilde n-1-m-\tilde\ell}{\tilde\ell} = %
  \binom{\tilde n-1-m-\tilde\ell}{\tilde\ell-1}$,
  the result follows.
\end{proof}

\begin{corollary}
\label{cor:Cno1reg}
  The number of quasi-hereditary $n$-CNakayama algebras without $1$-re\-gu\-lar simple modules equals the number of $n$-periodic Dyck paths with global shift~$0$ without $1$-rises.
  Explicitly, this number is
  \[
    \frac{1}{n}\sum_{k\divides n}\phi(k)\sum_{m=1}^{n/k-1}\binom{n/k}{m}\binom{n/k-m-1}{m-1}.
  \]
\end{corollary}
\begin{proof}
  Let $R(x) = x^2 + x R(x) + R(x)^2$ be the generating function for prime Dyck paths without $1$-rises and let $S(x)=R(x)/x$.
  Let $\tilde n=n/k$ and $\tilde\ell=\ell/k$.
  According to the cycle construction, we have to compute $[x^{\tilde n}]\log\frac{1}{1- x S(x)}$.
  Using the expansion of the logarithm, and Lagrange inversion with $S^{(-1)}(x) = \frac{x}{1+x+x^2}$, we obtain
  \begin{align*}
    [x^{\tilde n}]\log\frac{1}{1- x S(x)} %
    &= \sum_{m=1}^{n-1} [x^{\tilde n-m}] \frac{S(x)^m}{m} \\%
    &= \sum_{m=1}^{n-1} \frac{1}{\tilde n-m}[x^{\tilde n-2m}] (1+x+x^2)^{\tilde n - m} \\%
    \left(\substack{\text{substitute $n-m$ for $m$}\\\text{and $1/x$ for $x$}}\right) %
    &= [x^{\tilde n}] \sum_{m=1}^{n-1}\frac{1}{m} (1+x+x^2)^{m} \\%
    \left(\substack{\text{expand $\big(1+x(1+x)\big)$}}\right)
    &= \sum_{m=1}^{\tilde n-1}\frac{1}{m} \sum_{j=1}^{\tilde n} \binom{m}{j}\binom{j}{n-j}\\%
    &= \sum_{j=1}^{\tilde n}\binom{j}{\tilde n-j}\frac{1}{j}\sum_{m=1}^{\tilde n-1}\binom{m-1}{j-1}\\%
    \left(\substack{\text{\lq hockey stick identity\rq}}\right)
    &= \sum_{j=1}^{\tilde n}\binom{j}{\tilde n-j}\frac{1}{j}\binom{\tilde n-1}{j}.%
  \end{align*}
  A rearrangement of the final expression yields the claim.
\end{proof}

\begin{corollary}
\label{cor:Cno2reg}
  The number of quasi-hereditary $n$-CNakayama algebras without $2$-re\-gu\-lar simple modules equals the number of $n$-periodic Dyck paths with global shift~$0$ without $2$-hills, other than the path with constant area sequence.
  Explicitly, this number is
  \[
    \sum_{k\divides n}\frac{\phi(k)}{k}\sum_{m=0}^{\lfloor\frac{n}{2k}\rfloor}\frac{(-1)^m}{n/k-m}\binom{2n/k-3m-1}{n/k-m-1}-1.
  \]
\end{corollary}
\begin{proof}
  Let $D(x)=1+x D(x)^2$ be the generating function for all Dyck paths and let $H(x) = x D(x) - x^2$ be the generating function for prime Dyck paths without the $2$-hill.
  Let $\tilde n=n/k$ and $\tilde\ell=\ell/k$.
  According to the cycle construction, we have to compute $[x^{\tilde n}]\log\frac{1}{1- H(x)}$.
  Note that $\frac{1}{1- H(x)} = D(x)\frac{1}{1+x^2D(x)}$, we will compute the coefficient in the logarithm of these two factors separately.
  For the first factor, using that the compositional inverse of $D(x)-1$ is $\frac{x}{(1+x)^2}$, we obtain
  \begin{equation*}\label{eq:one}
    [x^{\tilde n}]\log\left(1+(D(x)-1)\right) = \frac{1}{\tilde n}[x^{\tilde n-1}](1+x)^{2\tilde n-1} = \frac{1}{\tilde n}\binom{2\tilde n-1}{\tilde n-1}.
  \end{equation*}
  For the second factor, we expand the logarithm and use that the compositional inverse of $x D(x)$ is $x(1-x)$:
  \begin{align*}
    [x^{\tilde n}]\log\left(\frac{1}{1+x^2 D(x)}\right)%
    &= \sum_{m\geq 1}\frac{(-1)^m}{m} [x^{\tilde n-m}] \big(xD(x)\big)^m \\%
    &= \sum_{m=1}^{\tilde n-1}\frac{(-1)^m}{\tilde n-m} [x^{\tilde n-2m}](1-x)^{m-\tilde n} \\%
    &= \sum_{m=1}^{\tilde n-1}\frac{(-1)^m}{\tilde n-m} \binom{2\tilde n-3m-1}{\tilde n-m-1}.
  \end{align*}
  We now note that we can extend the sum to $m=0$, and the additional term is precisely the expression computed in \Cref{eq:one}.
\end{proof}
Using \Cref{thm:enomoto}, \Cref{thm:LKC} also gives a sharp upper bound for the number of exact structures on the category of finitely generated projective modules for $n$-CNakayama algebras.

\begin{corollary}
\label{bound2regcnak}
  An $n$-CNakayama algebra has at most $\lfloor\frac{n}{2}\rfloor$ $2$-regular simple modules and thus at most $2^{\lfloor\frac{n}{2}\rfloor}$ exact structures on the category of finitely generated projective modules.
  This bound is sharp.
\end{corollary}
\begin{proof}
  A periodic Dyck path of semilength $n$ has at most $\lfloor\frac{n}{2}\rfloor$ $2$-hills.  Following \Cref{thm:LKC}, the bound is thus obtained for the $n$-CNakayama algebras with Kupisch series $[2,3,\dots,2,3]$ if $n$ is even, and, for example, $[2,3,\dots,2,3,2,2]$ if $n$ is odd.
\end{proof}

\section{Nakayama algebras of global dimension one and two}
\label{sec:globaldim}

The \Dfn{global dimension} $\gd(A)$ of an algebra~$A$ is the maximal projective dimension of a simple module, see for example \cite[Proposition~I.5.1]{ARS}.
In this section we consider Nakayama algebras of global dimension at most two.
\begin{definition}
  The \Dfn{height} of a (possibly periodic) Dyck path is the maximal entry in its area sequence minus one.
\end{definition}
\begin{theorem}
  \label{gldim2chara}
  An~$n$-Nakayama algebra~$A$ has global dimension $1$ if and only if it has Kupisch series $[n,\dots,1]$ corresponding to the unique Dyck path without valleys.  Any other~$n$-LNakayama algebra has global dimension $2$ if and only if for all $i$ such that $S_i$ is non-projective, we have
  \[
    c_{i+1}+1 \in \{c_i, c_{i+c_i} + c_i\},
  \]
  i.e., if and only if all valleys of the corresponding (possibly periodic) Dyck path are rectangles.

  If~$A$ is an~$n$-LNakayama algebra or the (possibly periodic) Dyck path $D$ corresponding to~$A$ is in $\DyckP_n\cup\DyckH_n$, the Nakayama algebra has global dimension~$2$ if and only if $\LK(D)$ has height~$2$.

  Moreover, $(n+1)$-LNakayama algebras of global dimension $2$ with exactly~$\ell$ simple modules of projective dimension $2$ are in bijection with subsets of $\{1,\dots,n\}$ of cardinality $2\ell$, counted by $\binom{n}{2\ell}$.

 $n$-CNakayama algebras of global dimension~$2$ with exactly~$\ell$ simple modules of projective dimension $2$ are in bijection with subsets of $\{0,\dots,n-1\}$ of cardinality $2\ell$ up to rotation by pairs\footnote{\OEIS{A052823}}.  Explicitly, this number is
  \[
    \frac{2}{n}\sum_{k\divides\gcd(\ell,n)} \phi(k)\binom{n/k}{2\ell/k}.
  \]
\end{theorem}

\begin{proof}
  The global dimension of a Nakayama algebra equals the maximal projective dimension of a simple module~$S_i$.  Thus, the characterization in terms of the Kupisch series is an immediate consequence of \Cref{smallprojdim}.  The reformulation in terms of rectangles is immediate from the definition.

  It remains to describe the claimed bijections.
  Let~$A$ be an $(n+1)$-LNakayama algebra whose simple modules of projective dimension $2$ are $S_{i_1},\dots,S_{i_\ell}$, corresponding to rectangles of $D$ with $x$-coordinates $1 < i_1+1 < \dots < i_\ell+1 \leq n$.
  Mirroring $D$ below the main diagonal, $\BJS$ puts crosses into the cells of the corresponding valleys, with top left coordinates $(j_1,i_1+1),\dots,(j_\ell,i_\ell+1)$.
Then, working from right to left, $\BJS$ puts crosses into the cells on the main diagonal, with top-left coordinates $(0,1),\dots,(i_1-1,i_1)$.

  Because there is a rectangle with $x$-coordinate $i_1+1$, the next valley at $(j_2, i_2+1)$ has $y$-coordinate strictly larger than $j_1+1$.  Thus, there are no crosses corresponding to valleys of $D$ with $y$-coordinates $i_1+2,\dots,j_1+1$.  Therefore, $\BJS$ puts crosses into the cells on the super diagonal with top-left coordinates $(i_1,i_1+2),\dots,(j_1-1,j_1+1)$.

  The process then continues by putting crosses into the cells on the main diagonal again, with top-left coordinates $(j_1+1,j_1+2),\dots,(i_2-1,i_2)$, and so on.  It is not hard to see that any Dyck path of height~$2$ can be obtained this way.

  Similarly, one finds that mapping $D$ to the set
  \[
    i_1+1 < j_1+1 < i_2+1 < \dots < i_\ell+1 < j_\ell+1
  \]
  is a bijection with subsets of $\{1,\dots,n\}$ of size $2\ell$.
\end{proof}

\begin{example}
\label{ex:globaldim2}
  \begin{figure}
    \centering
    \resizebox{.55\textwidth}{!}{
      \begin{tikzpicture}[scale=0.75]
        \draw[dotted] (0, 0) grid (12, 12);
        \draw[dotted] (0, 0) -- (12, 12);
        \drawPath%
        {1,1,1,0,1,1,1,1,1,1,1,0,0,0,0,0,0,0,1,1,0,0,0,0} 
        {0}{black}{(0,0)}{2}
        \drawArea{5,4,10,9,8,7,6,5,4,4,3,2,1}{4}{lightgrey}{0}{0}
        \drawLabel{c}{5,4,10,9,8,7,6,5,4,4,3,2,1}
        {\coc}{}
        \draw[<->, rounded corners=1, line width=2pt, dashed] (1, 6.5) -- (7.5,6.5) -- (7.5,10);
        \draw[<->, rounded corners=1, line width=2pt, dashed] (0,-0.5) -- (0.5,-0.5) -- (0.5,3);
      \end{tikzpicture}
    }
    \caption{\label{fig:gdtwoexample}The Dyck path in \Cref{ex:globaldim2} with rectangles at coordinates $(2,4)$ and coordinates $(9,11)$.}
  \end{figure}
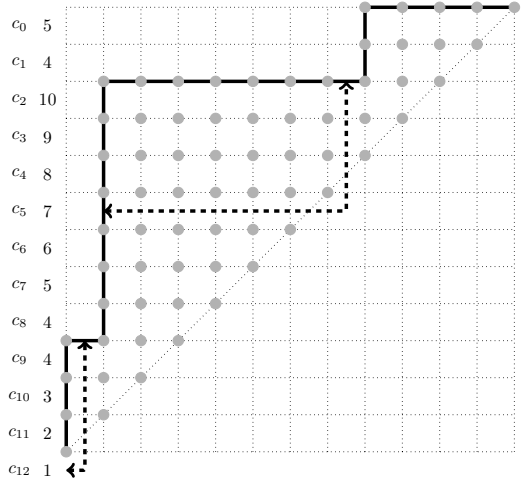
  The $13$-LNakayama algebra with Kupisch series
  \[
    [5,4,10,9,8,7,6,5,4,4,3,2,1]
  \]
  has global dimension~$2$ and its simple modules $S_i$ have projective dimension~$2$ exactly for indices $i \in \{1,8\}$, where we compute
  \[
  c_2+1-c_1 = 7 = c_5 = c_{1+c_1}, \quad c_9+1-c_8 = 1 = c_{12} = c_{8+c_8}.
  \]
  The corresponding Dyck path is shown in \Cref{fig:gdtwoexample}, and is sent to the set $\{2,5,9,12\}$.
  To see how to recover the path from this set $\{j_1=2,j_2=5,j_3=9,j_4=12\}$, observe that we obtain that $c_{i+1} + 1 = c_i$ for all~$i$ except
  \[
    i \in \{j_1-1,j_3-1\} = \{ 1,8\},
  \]
  and
  \[
    c_1 = j_2-(j_1-1) = 4, \quad c_8 = j_4-(j_3-1) = 4.
  \]
  This in turn uniquely determines the Kupisch series as given.
\end{example}

Combining \Cref{regchara}\eqref{eq:2regchara} with \Cref{gldim2chara}, we thus obtain the following description of $2$-regular simple modules of Nakayama algebras of global dimension~$2$.

\begin{corollary} \label{gldim2corollary}
  Let~$A$ be an~$n$-Nakayama algebra of global dimension $2$.
  \begin{enumerate}
  \item if~$A$ is a CNakayama algebra, $S_i$ is $2$-regular if and only if $c_i=2$.
  \item if~$A$ is an LNakayama algebra, $S_{n-2}$ and $S_{n-1}$ are never $2$-regular, and $S_i$ is $2$-regular for $i < n-2$ if and only if $c_i=2$.
  \end{enumerate}
\end{corollary}
\begin{proof}
  Suppose that $c_i=2$.  It follows that the (possibly periodic) Dyck path $D$ corresponding to~$A$ has a valley at $(i+1,i+1)$ and in particular $d_{i+2}=2$.  By \Cref{gldim2chara}, $c_{i+1}=1$ or $c_{i+1} = c_{i+2} + 1$.  In the former case,~$A$ is an LNakayama algebra and $i=n-1$.  It remains to show that in the latter case, we also have $d_{i+1} = d_i + 1$, as required by \Cref{regchara}\eqref{eq:2regchara}.  Suppose on the contrary that $d_{i+1} \leq d_i$, so that $D$ has a valley at $(j+1, i)$ for some $j$.  This valley cannot belong to a rectangle with $x$-coordinate $j+1$, because then the next valley should have $y$-coordinate strictly larger than $i+1$.  Thus, by \Cref{gldim2chara}, $c_{j+1} + 1 = c_j$, contradicting the assumption that there is a valley at $(j+1, i)$.
\end{proof}

Observe that the conclusion in the previous corollary does not hold in general for Nakayama algebras of higher global dimension as the next example shows.
\begin{example}
  Let~$A$ be the LNakayama algebra with Kupisch series $[2,2,2,1]$.
  Then~$A$ has global dimension $3$ and does not have any simple $2$-regular modules.
  The LNakayama algebra with Kupisch series $[2,3,2,2,2,1]$ also has global dimension $3$, and $S_0$ is its only simple $2$-regular module.
\end{example}

Let us now use our preceding results to classify LNakayama algebras of global dimension at most two that satisfy the restricted Gorenstein condition.  An algebra~$A$ with finite global dimension~$k \geq 0$ satisfies the \Dfn{restricted Gorenstein condition} if $k=0$ or if every simple left and right module with projective dimension~$k$ is~$k$-regular, see for example~\cite[Definition~1.3]{Iy}.

We say that a Dyck path is a \Dfn{bounce path} if it is of the form $h^{a_1}v^{a_1} \dots h^{a_\ell}v^{a_\ell}$ for positive integers $a_1,\dots,a_\ell$.
Observe that this is the case if and only if all its valleys are of the form $(i,i)$.
Moreover, an LNakayama algebra has a associated Dyck path that is a bounce path if and only if its Kupisch series is of the form
\[
  [a_1+1,\dots,2,\ a_2+1,\dots,2,\ \dots,\ a_\ell+1,\dots,2,1].
\]

\begin{theorem}
\label{thm:resGor}
  Let~$A$ be an~$n$-Nakayama algebra and let~$D$ be the associated Dyck path.
  Then~$A$ has global dimension at most~$2$ and satisfies the restricted Gorenstein condition if and only if~$D$ is a bounce path and has no $1$-hill after position~$0$ and before position $n-1$.

  Similarly, let~$A$ be an~$n$-CNakayama algebra.  Then~$A$ has global dimension at most~$2$ and satisfies the restricted Gorenstein condition if and only if~$D$ is a bounce path and has no $1$-hills.
\end{theorem}

\begin{proof}
  Suppose that $D$ is a bounce path without $1$-hills after position $0$ and before position $n-1$.  Then all valleys of $D$ belong to rectangles, so by \Cref{gldim2chara}, the global dimension of~$A$ is at most $2$.  Since all rectangles of $D$ are returns, the corresponding simple modules are all $2$-regular by \Cref{thm:Dyck-regular}\eqref{it:Dyck-2-regular}.

  Conversely, suppose that~$A$ has global dimension at most $2$, but $D$ has a rectangle with $x$-coordinate $i+1$ which is not a return.  Then $c_i>2$, so $S_i$ is not $2$-regular.

  To conclude that~$A$ satisfies the restricted Gorenstein condition, it remains to recall that the simple left modules of~$A$ are the simple right modules of the opposite algebra, corresponding to the reversed Dyck path.  The above reasoning applies verbatim.
\end{proof}

\begin{corollary}
\label{cor:resGor}
  The number of $(n+1)$-LNakayama algebras of global dimension at most~$2$ that satisfy the restricted Gorenstein condition equals the Fibonacci number\footnote{\OEIS{A000045}} $F(n+1)$, counting subsets of $\{1,2,\dots,n-1\}$ that contain no consecutive integers.
  Explicitly, this number is given by the recurrence
  \[
    F(n+2) = F(n)+F(n+1)
  \]
  with initial conditions $F(1)=F(2)=1$.
\end{corollary}

\begin{proof}
  A bounce path of semilength~$n$ can be identified with the subset of $\{1,\dots,n-1\}$ given by the positions of its valleys.  Under this identification, a $1$-hill at a position between $1$ and $n-2$ corresponds to two consecutive numbers in the given subset, which implies the claim.
\end{proof}

The analogous result for CNakayama algebras is as follows.
\begin{corollary}
\label{cor:resGorCNakayama}
  The number of~$n$-CNakayama algebras of global dimension~$2$ that satisfy the restricted Gorenstein condition equals the number of cyclic compositions of~$n$ into parts of size at least $2$\footnote{\OEIS{A032190}}.  Explicitly, this number is
  \[
    \frac{1}{n}\sum_{k\divides n} \phi(n/k)\big(F(k-1)+F(k+1)\big)-1
  \]
  where $F$ is the Fibonacci number defined above.
\end{corollary}

We remark that by \cite[Proposition~1.4]{Iy}, the class of Nakayama algebras with global dimension at most $2$ satisfying the restricted Gorenstein condition coincides with the class of Nakayama algebras of global dimension at most $2$ that are $2$-Gorenstein. For the general enumeration of $2$-Gorenstein LNakayama algebras we refer to the recent article \cite{RS2017}.

\bibliographystyle{amsalpha}

\end{document}